\documentclass[11pt,letterpaper,reqno]{amsart}
\usepackage{fullpage}
\usepackage[top=2cm, bottom=4.5cm, left=2.5cm, right=2.5cm]{geometry}
\usepackage{amsmath,amsthm,amsfonts,amssymb,amscd}
\usepackage{geometry,tikz,tikz-cd}
\usepackage{empheq}
\usepackage{lipsum}
\usepackage{lastpage}
\usepackage{thmtools}
\usepackage{enumerate}
\usepackage[shortlabels]{enumitem}
\usepackage{fancyhdr}
\usepackage{mathrsfs}
\usepackage{xcolor}
\usepackage{graphicx}
\usepackage{listings}
\usepackage{bbm}
\usepackage{array}
\usepackage{hyperref}
\usepackage[makeroom]{cancel}

\usepackage[utf8]{inputenc}
\usepackage{comment}

\usepackage{makecell}

\usepackage[all]{xy}
\xyoption{matrix}
\xyoption{arrow}

\usetikzlibrary{arrows,decorations.pathmorphing,backgrounds,positioning,fit,petri}

\tikzset{help lines/.style={step=#1cm,very thin, color=gray},
help lines/.default=.5} % draws a grid spaced #1 cm
\tikzset{thick grid/.style={step=#1cm,thick, color=gray},
thick grid/.default=1} % draws a grid spaced #1 cm

\mathtoolsset{showonlyrefs=true}
\numberwithin{figure}{section}
\numberwithin{table}{section}

\theoremstyle{definition}

\theoremstyle{plain}
\newcommand{\thistheoremname}{}
\newtheorem*{genericthm*}{\thistheoremname}
\newenvironment{namedthm*}[1]
  {\renewcommand{\thistheoremname}{#1}%
   \begin{genericthm*}}
  {\end{genericthm*}}

\hypersetup{%
  colorlinks=true,
  linkcolor=blue,
  citecolor=blue,
  linkbordercolor={0 0 1}
}

\DeclareMathOperator{\ZZ}{\mathbb{Z}}

\newcommand{\lrabs}[1]{\left\lvert #1 \right\lvert}
\newcommand{\lrp}[1]{\left(#1\right)}
\newcommand{\lrb}[1]{\left[#1\right]}

\allowdisplaybreaks
\newtheorem{theorem}{Theorem}[section]

\newtheorem{lemma}[theorem]{Lemma}

\newtheorem{mytable}[theorem]{Table}
\newtheorem{proposition}[theorem]{Proposition}
\newtheorem{conjecture}[theorem]{Conjecture}
\theoremstyle{definition}

\theoremstyle{definition}

\theoremstyle{remark}

\numberwithin{equation}{section}

\usepackage{latexsym}
\usepackage{amssymb}

\newcommand{\ben}{\begin{equation}}
\newcommand{\een}{\end{equation}}

\newcommand{\SL}[1]{\ensuremath{{\mathrm {SL}_{ #1 }}}}

\makeatletter
\NewDocumentCommand{\sump}{e{_}}
 {%
  \DOTSB
  \mathop{\IfNoValueTF{#1}{\sump@{}}{\sump@{#1}}}%
  \nolimits
 }
\newcommand{\sump@}[1]{\mathpalette\sump@@{#1}}
\newcommand{\sump@@}[2]{%
  \ifx#1\displaystyle
    {\sump@display{#2}}%
  \else
    \sum@\nolimits'_{#2}%
  \fi
}
\newcommand{\sump@display}[1]{%
  \sbox\z@{$\m@th\displaystyle\sum@\nolimits'$}%
  \sbox\tw@{$\m@th\displaystyle\sum@\limits_{#1}$}%
  \sbox\@tempboxa{$\m@th\displaystyle'$}
  \mathop{\sum@\nolimits' \kern-\wd\@tempboxa}\limits_{#1}%
  \ifdim\wd\z@>\wd\tw@
    \kern\dimexpr\wd\z@-\wd\tw@\relax
  \fi
}
\makeatother

\DeclareMathOperator{\Tr}{Tr}

\DeclareMathOperator{\lcm}{lcm}
\setlist[enumerate]{leftmargin=*,widest=0}
\setlist[itemize]{leftmargin=*,widest=0}
\makeatletter
\def\subsection{\@startsection{subsection}{2}%
  \z@{.5\linespacing\@plus.7\linespacing}{.3\linespacing}%
  {\normalfont\bfseries}}

\def\subsubsection{\@startsection{subsubsection}{3}%
  \z@{.5\linespacing\@plus.7\linespacing}{.3\linespacing}%
  {\normalfont\bfseries}}
\makeatother

\begin{document}

\title{Signs of the Second Coefficients of Hecke Polynomials}


\author[E. Ross]{Erick Ross}
\address[E. Ross]{School of Mathematical and Statistical Sciences, Clemson University, Clemson, SC, 29634}
\email{erickr@clemson.edu}

\author[H. Xue]{Hui Xue}
\address[H. Xue]{School of Mathematical and Statistical Sciences, Clemson University, Clemson, SC, 29634}
\email{huixue@clemson.edu}

\begin{abstract}
    Let $T_m(N, k, \chi)$ be the $m$-th Hecke operator of level $N$, weight $k \geq 2$, and nebentypus $\chi$, where $N$ is coprime to $m$.  We first show that for any given $m \geq 1$, the second coefficient of the characteristic polynomial of $T_m(N, k, \chi)$ is nonvanishing for all but finitely many triples $(N,k,\chi)$. Furthermore, for $\chi$ trivial and any fixed $m$, we determine the sign of the second coefficient for all but finitely many pairs $(N,k)$. Finally, for $\chi$ trivial and $m=3,4$, we compute the sign of the second coefficient for all pairs $(N,k)$.

\end{abstract}

\maketitle

\section{Introduction} \label{sec:intro}
Let $k\ge2$, $N\ge1$ be integers and let $\chi$ be a Dirichlet character modulo $N$ such that $\chi(-1) = (-1)^k$. The space of cuspforms of level $N$, weight $k$, and nebentypus $\chi$ is denoted by $S_k(\Gamma_0(N),\chi)$ \cite[Section 7.2]{cohen-stromberg}. For $m \geq 1$, let $T_m(N,k,\chi)$ be the $m$-th Hecke operator  on $S_k(\Gamma_0(N),\chi)$ \cite[Chapter 10]{cohen-stromberg}.  When the character $\chi$ is trivial, we will drop $\chi$ and simply write $T_m(N,k)$ and $S_k(\Gamma_0(N))$, respectively. 
Several interesting questions about these Hecke operators $T_m(N, k, \chi)$ have been studied.  For instance, let
$$\mathit{\Delta}(z)=q\prod_{n\geq1}(1-q^n)^{24}=\sum_{n\geq1}\tau(n)q^n$$
be the discriminant function,
which is the unique normalized cuspform of weight $12$, level one, and trivial nebentypus.  Lehmer \cite{Lehmerconjecture} conjectured that $\tau(m)\neq0$ for all $m\geq1$. Let $\Tr T_m(N, k)$ denote the trace of $T_m(N, k)$ on the space $S_k(\Gamma_0(N))$. The Lehmer Conjecture can then be reinterpreted as follows: $\Tr T_m(1, 12) \neq 0$ for all $m\geq 1$. More broadly, Rouse \cite{rouse} gave the Generalized Lehmer Conjecture, which predicts that $\Tr T_m(N,k)\neq0$ for even $k\geq 16$ or $k=12$ and $\gcd(m,N)=1$. He also proved this result for $m=2$. Recently, the nonvanishing of $\Tr T_{3}(1,k)$ was also established in \cite{chiriac}.

Now, let us write the characteristic polynomial for $T_m(N,k,\chi)$, the so-called Hecke polynomial,
as 
\begin{align}
T_m(N,k,\chi)(x)=x^n - a_1(m,N,k,\chi) x^{n-1} + a_2(m,N,k,\chi) x^{n-2} - \cdots + (-1)^n a_n(m,N,k,\chi),
\end{align}
where $n=\dim S_k(\Gamma_0(N),\chi)$.
Here, we will refer to $a_i(m,N,k,\chi)$ as the $i$-th coefficient of the Hecke polynomial. 
To ease notation we will simply write $a_i(m,N,k)$
if $\chi$ is trivial.
Using this notation, the Generalized Lehmer Conjecture concerns the nonvanishing of the first coefficient $a_1(m,N,k)$ of $T_m(N,k)(x)$. One may also consider the nonvanishing of the other coefficients, 
% of $T_m(N,k,\chi)(x)$
in particular the second coefficient, $a_2(m,N,k,\chi)$.
 Most recently, for trivial characters $\chi$, Clayton et al. \cite[Theorems 1.1 and 1.3]{clayton-et-al} computed the complete list of pairs $(N,k)$ for which the second coefficient of $T_2(N,k)(x)$ vanishes. In this paper,  we shall first extend the results of \cite{clayton-et-al} to study the nonvanishing of $a_2(m,N,k,\chi)$ for general $m$, $N$, $k$, and $\chi$. More precisely, we have the following result.

\begin{theorem} \label{thm:main}
    Let $m \geq 1$ be fixed. Suppose $\gcd(N,m)=1$, $k \geq 2$, and $\chi$ is a Dirichlet character modulo $N$ such that $\chi(-1) = (-1)^k$. Then $a_2(m,N,k,\chi)$ vanishes for only finitely many triples $(N,k,\chi)$.
\end{theorem}

When $\chi$ is trivial, we also explicitly determine the sign of the second coefficient.

\begin{theorem} \label{thm:main-trivial-char}
    Let $m \geq 1$ be fixed. Suppose $\gcd(N,m)=1$ and $k \geq 2$ is even. \\
    $(1)$: If $m$ is not a perfect square, then $a_2(m,N,k)$ is negative for all but finitely many pairs $(N,k)$. \\
    $(2)$: If $m$ is a perfect square, then $a_2(m,N,k)$ is positive for all but finitely many pairs $(N,k)$. 
\end{theorem}

We show these two theorems by first expressing $a_2(m,N,k,\chi)$ in terms of traces of various Hecke operators (Lemma \ref{lem:a2-coeff-formula}). These traces can each be evaluated by the Eichler-Selberg trace formula \eqref{eqn:eichler-selberg-trace-formula}.  From this formula, we can then identify the dominant terms for these traces, coming from the Hecke operators with perfect square index (Lemma \ref{lem:main-lemma-perfect-square}).  This allows us to determine the asymptotic growth of $a_2(m,N,k,\chi)$, which then yields Theorems \ref{thm:main} and \ref{thm:main-trivial-char}.

In fact, our method is effective: for any given $m$, these exceptional pairs can be computed explicitly. As an illustratation of the two cases in Theorem \ref{thm:main-trivial-char},
when $\chi$ is trivial and $m=3,4$, we carry out the details to compute all the exceptional pairs.

\begin{theorem} \label{thm:T3}
    Suppose that $\gcd(N,3)=1$ and that $k \geq 2$ is even.
    Then  $a_2(3,N,k)$ is positive or zero only for the pairs $(N,k)$ given in Table \ref{table:T3-pairs-dim-a2}.
\end{theorem}

\begin{theorem} \label{thm:T4}
    Suppose that $\gcd(N,4)=1$ and that $k \geq 2$ is even.
    Then  $a_2(4,N,k)$ is negative or zero only for the pairs $(N,k)$ given in \cite[Table $m=4$]{ross-code}.
\end{theorem}

The paper is organized as follows. In Section \ref{sec:trace}, following the idea in \cite{clayton-et-al}, we express the second coefficient $a_2(m,N,k,\chi)$ in terms of traces of Hecke operators. We also state the Eichler-Selberg trace formula to compute these traces. Section \ref{sec:estimate} is preparatory and establishes estimates on certain terms in the Eichler-Selberg trace formula. In Section \ref{sec:proofofmain}, we prove Theorems \ref{thm:main} and \ref{thm:main-trivial-char}. In Sections \ref{sec:proofofT3} and \ref{sec:proofofT4} we apply the techniques developed in Section \ref{sec:proofofmain} to the cases of $m=3$ and $m=4$, and prove Theorems \ref{thm:T3} and \ref{thm:T4}, respectively. Section \ref{sec:discussion} discusses some related questions.

\section{Second coefficients in terms of traces}
\label{sec:trace}
Following \cite[Proposition 2.1]{clayton-et-al}, we first derive a formula for $a_2(m,N,k,\chi)$  in terms of traces of Hecke operators.
\begin{lemma} \label{lem:a2-coeff-formula}
    For convenience of notation, let $T_m$ denote $T_m(N,k,\chi)$. Then 
    \begin{align}     
    a_2(m,N,k,\chi) = \frac{1}{2} \lrb{ \left(\Tr T_m\right)^2 - \sum_{d \mid m} \chi(d) d^{k-1} \Tr T_{m^2/d^2} }. \end{align}
\end{lemma}
\begin{proof}
    Let $\lambda_1 \ldots \lambda_n$ be the eigenvalues of $T_m$. By the definition of characteristic polynomial, we have 
    \begin{align}
        a_2(m,N,k,\chi) 
        &= \sum_{1 \leq i < j \leq n} \lambda_i \lambda_j \\
        &= \frac12 \lrb{ \lrp{ \sum_{1 \leq i \leq n} \lambda_i }^2 - \sum_{1 \leq i \leq n} \lambda_i^2 } \\
        &= \frac12 \lrb{ (\Tr T_m)^2 - \Tr T_m^2 }. 
    \end{align}
    On the other hand, recall the following formula  \cite[Theorem~10.2.9]{cohen-stromberg} for Hecke operators: 
    $$ T_m^2 = \sum_{d\mid m} \chi(d) d^{k-1}  T_{m^2/d^2} .$$
    Thus, 
    \begin{align}
        a_2(m,N,k,\chi) &= \frac12 \lrb{ (\Tr T_m)^2 - \Tr T_m^2 }  \\
        &= \frac12 \lrb{ (\Tr T_m)^2 - \sum_{d \mid m} \chi(d) d^{k-1} \Tr T_{m^2/d^2} },
    \end{align}
as desired.
\end{proof}

Next, we state the Eichler-Selberg trace formula in order to give an explicit formula for the traces appearing in Lemma \ref{lem:a2-coeff-formula}.  Let $m \geq 1$, $N \geq 1$, $k \geq 2$, and $\chi$ be a Dirichlet character modulo $N$ such that $\chi(-1) = (-1)^k$. From \cite[pp.~370-371]{knightly-li}, and borrowing some notation from \cite[24.4.11]{cohen-stromberg}, the Eichler-Selberg trace formula is given by
\begin{equation} \label{eqn:eichler-selberg-trace-formula}
\Tr T_m(N,k,\chi) = A_{1,m} - A_{2,m} - A_{3,m} + A_{4,m} ,
\end{equation}
where
\begin{align}
    \label{eqn:A1m-formula}
    A_{1,m} &= \chi(\sqrt{m}) \frac{k-1}{12} \psi (N) m^{k/2-1},  \\
   \label{eqn:A2m-formula}
   A_{2,m} &= \frac{1}{2} \sum_{t^2 < 4m} U_{k-1}(t,m) \sum_{n}  h_w \lrp{ \frac{t^2-4m}{n^2} } \mu (t,n,m), \\
   \label{eqn:A3m-formula}
   A_{3,m} &= \frac{1}{2} \sum_{d|m} \min(d,m/d)^{k-1} \sum_{\tau} \phi ( \gcd(\tau, N / \tau )) \chi(y_{\tau}), \\
   \label{eqn:A4m-formula}
   A_{4,m} &= \begin{dcases} 
      \sum_{\substack{c|m \\ (N, m/c) = 1}} c & \text{if $k=2$ and $\chi = \chi_0$},\\
      0 & \text{if $k>2$ or $\chi \neq \chi_0$} .\\
   \end{dcases}
\end{align}

Here, we have the following notation.

\begin{itemize}
  \item $\chi(\sqrt{m})$ is interpreted as $0$ if $m$ is not a perfect square. 
  
  \item $\psi (N) = \lrb{\, \SL{2}(\ZZ) \, : \, \Gamma_0 (N) \,} = N \prod_{p|N} \lrp{ 1 + \frac{1}{p} }$.
  
  \item The outer summation in $A_{2,m}$ runs over all $t \in \ZZ$ such that $t^2 < 4m$. Note that the terms corresponding to $t=t_0$ and $t=-t_0$ coincide.

  \item $U_{k-1}(t,m)$ denotes the Lucas sequence of the first kind. In particular, $U_{k-1}(t,m) = \frac{\rho^{k-1} - \bar{\rho}^{k-1}}{\rho - \bar{\rho}}$ where $\rho, \bar{\rho}$ are the two roots of the polynomial $X^2-tX+m$.
  
  \item The inner summation in $A_{2,m}$ runs through all positive integers $n$ such that $n^2 \, | \, (t^2 -4m)$ and $\frac{t^2-4m}{n^2} \equiv 0,1 \pmod{4}$. 

  \item $h_w \lrp{ \frac{t^2-4m}{n^2} }$ is the weighted class number of the imaginary quadratic order with discriminant $\frac{t^2-4m}{n^2}$. This is the usual class number, divided by $2$ (respectively $3$) if the discriminant is $-4$ (respectively $-3$). For our purposes, the first few of them are given explicitly in Table \ref{table:weighted-class-numbers} below. 

  \item $\displaystyle \mu (t,n,m) = \frac{\psi (N)}{\psi (N/N_n)} \sideset{}{'} \sum_{c \!\! \mod N} \chi(c)$, where $N_n = \gcd(N,n)$, and the primed summation runs through all elements $c$ of $(\ZZ / N \ZZ)^{\times}$ which lift to solutions of $c^2-tc+m \equiv 0 \pmod{NN_n}$.

  \item The outer summation for $A_{3,m}$ runs through all positive divisors $d$ of $m$. Note that the terms corresponding to $d=d_0$ and $d=m/d_0$ coincide.
  
  \item The inner summation for $A_{3,m}$ runs over all positive divisors $\tau$ of $N$ such that $\gcd(\tau, N/\tau)$ divides $\gcd(N/N_{\chi}, d-m/d)$. Here $N_{\chi}$ is the conductor of $\chi$.

  \item $\phi$ is the Euler totient function.

  \item $y_{\tau}$ is the unique integer modulo $\lcm(\tau, N/\tau)$ determined by the congruences $y_{\tau} \equiv d \pmod{\tau}$ and $y_{\tau} \equiv \frac{m}{d} \pmod{\frac{N}{\tau}}$.

  \item $\chi_0$ denotes the trivial character modulo $N$.

  \item Throughout, remember that $\chi$ is a character modulo $N$, so $\chi(a) = 0$ if $\gcd(a,N) > 1$, even in the trivial character case.
  
\end{itemize}

\begin{mytable}[{Weighted class numbers; \cite[p.~345]{knightly-li}, \cite[A014600]{oeis}}] \label{table:weighted-class-numbers} 
$$
\begin{array}{|c|c|c|c|c|c|c|c|c|c|c|c|}
\hline
 n & -3 & -4 & -7 & -8 & -11 & -12 & -15 & -16 & -19 & -20 & -23 \\
\hline
 h_w(n) & \frac{1}{3} & \frac{1}{2} & 1 & 1 & 1 & 1 & 2 & 1 & 1 & 2 & 3 \\
 \hline \hline
 n & -24 & -27 & -28 & -31 & -32 & -35 & -36 & -39 & -40 & -43 & -44 \\
 \hline
 h_w(n) & 2 & 1 & 1 & 3 & 2 & 2 & 2 & 4 & 2 & 1 & 3 \\   
 \hline \hline 
 n & -47 & -48 & -51 & -52 & -55 & -56 & -59 & -60 & -63 & -64 & -67 \\
 \hline
 h_w(n) & 5 & 2 & 2 & 2 & 4 & 4 & 3 & 2 & 4 & 2 & 1 \\ 
 \hline
\end{array}
$$

\end{mytable}

\section{Estimates on terms in the trace formula} \label{sec:estimate}
In this section we give estimates on the $A_{i,m}$ trace terms \eqref{eqn:A1m-formula}, \eqref{eqn:A2m-formula}, \eqref{eqn:A3m-formula}, and \eqref{eqn:A4m-formula}.
First, we introduce some arithmetic functions  that will be used to express these estimates. 
\begin{lemma} \label{lem:omega-over-psi-bounds}
    Recall that $\psi(N) = N \prod_{p|N} \lrp{ 1 + \frac{1}{p} }$, and let $\omega(N)$ denote the number of distinct prime divisors of $N$. Define
    \begin{align*}
        \theta_1(N) := \frac{2^{\omega(N)} \sqrt{N} }{\psi(N)}, 
        \qquad 
        \theta_2(N) := \frac{(2^{\omega(N)})^2}{\psi(N)}, 
        \qquad
        \theta_3(N) := \frac{2^{\omega(N)}}{\psi(N)}.
    \end{align*}
    Then each $\theta_i(N) \rightarrow 0$ as $N \rightarrow \infty$. \\
    In particular, we have the bounds given in the following table.
    $$     
\begin{array}{|c|c|c|c|c|c|c|c|}
\hline
 \displaystyle N \geq & 1 & 43 & 571 & 8,\!800 & 150,\!000 & 2,\!700,\!000 & 63,\!000,\!000  \\
\hline
\displaystyle \theta_1(N) \leq & 1.00 & 0.465 & 0.257 & 0.133 & 0.0607 & 0.0265 & 0.0106  \\
\hline 
\displaystyle \theta_2(N) \leq & 1.34 & 0.445 & 0.149 & 0.0424 & 0.00941 & 0.00189 & 0.000314  \\
\hline 
\displaystyle \theta_3(N) \leq & 1.00 & 0.0556 & 0.00926 & 0.00133 & 0.000147 & 0.000015 &  0.000015 \\
\hline
\end{array}
$$

\end{lemma}
\begin{proof}
    Note that every prime other than $2,3,5,7$ is $\geq 8$. Thus $\omega(N) \leq 4 + \log_8(N)$ and so $2^{\omega(N)} \leq 2^{4 + \log_8(N)} \leq 16 \cdot N^{1/3}$. Since $\psi(N) \geq N$, it is clear that
    $$\theta_1(N) = \frac{2^{\omega(N)} \sqrt{N} }{\psi(N)}, \ \ \ \theta_2(N) = \frac{(2^{\omega(N)})^2}{\psi(N)}, \ \ \ \theta_3(N) = \frac{2^{\omega(N)}}{\psi(N)} \ \ \longrightarrow 0 \ \text{ as } \ N \longrightarrow \infty.$$
    To prove the specific numerical bounds given above, we will first show that
    \begin{equation} \label{eqn:three-numerical-theta-bounds}
        \theta_1(N) \leq 0.0106,\ \theta_2(N) \leq 0.000314, \ \theta_3(N) \leq 0.000015
    \end{equation}
    for all $N \geq 584,\!000,\!000$. Then we will verify each of the claimed bounds in the table by exhaustive computer check over all $N < 584,\!000,\!000$.
    
    Let $p_n$ denote the $n$-th prime number, and let $P_n := p_1\cdots p_n$.
    % denote the product of the first $n$ primes. 
    For all $N$ with $\omega(N) \geq 9$, we show that $\theta_i(N) \leq \theta_i(P_9)$.
    For such $N$, let $N=q_1^{e_1} \cdots q_m^{e_m}$ be its prime factorization. Then 
    \begin{align*}
        \frac{\psi(N)}{\sqrt{N}} &= \frac{ (q_1+1)q_1^{e_1-1} \cdots (q_m+1)q_m^{e_m-1} }{q_1^{e_1/2} \cdots q_m^{e_m/2}} \\
        &\geq \frac{ (q_1+1) \cdots (q_m+1) }{q_1^{1/2} \cdots q_m^{1/2}} \\
        &\geq \frac{ (p_1+1) \cdots (p_m+1) }{p_1^{1/2} \cdots p_m^{1/2}} \qquad \qquad \qquad \lrp{ \text{since $\frac{x+1}{x^{1/2}}$ is increasing for $x \geq 1$} } \\
        % &= \frac{ (p_{1}+1) \cdots (p_9+1) }{p_{1}^{1/2} \cdots p_9^{1/2}} \cdot  \frac{ (p_{10}+1) \cdots (p_m+1) }{p_{10}^{1/2} \cdots p_m^{1/2}}  \\
        &= \frac{\psi(P_9)}{\sqrt{P_9}}  \frac{ (p_{10}+1) \cdots (p_m+1) }{p_{10}^{1/2} \cdots p_m^{1/2}}  \\
        &\geq \frac{\psi(P_9)}{\sqrt{P_9}}  2^{m-9} .
    \end{align*}
    This means that
    $\displaystyle
        \theta_1(N) = \frac{2^m \sqrt{N}}{\psi(N)}  
        \leq \frac{2^m \sqrt{P_9}}{2^{m-9} \psi(P_9)}  = \theta_1(P_9) \leq 0.0106
    $.
    
    By an identical argument,
    $\displaystyle
        \psi(N) \geq 4^{m-9} \psi(P_9)
    $,
    which means that
    $\displaystyle
        \theta_2(N) \leq \theta_2(P_9) \leq 0.000314 
    $
    and
    $\displaystyle
         \theta_3(N) \leq \theta_3(P_9) \leq 0.000015 
    $.
    This verifies the three bounds in \eqref{eqn:three-numerical-theta-bounds} for all $N$ with $\omega(N) \geq 9$. 
    
    For $N$ with $\omega(N) \leq 8$ and $N \geq 584,\!000,\!000$, we have
    \begin{alignat}{3}
        \theta_1(N) &= \frac{2^{\omega(N)} \sqrt{N}}{\psi(N)} &&\leq \frac{2^8 \sqrt{N}}{N} &&\leq 0.00915 , \\
        \theta_2(N) &= \frac{(2^{\omega(N)})^2}{\psi(N)} &&\leq \frac{2^{16}}{N} &&\leq 0.000314 , \\
        \theta_3(N) &= \frac{2^{\omega(N)}}{\psi(N)} &&\leq \frac{2^8}{N} &&\leq 0.000015 .
    \end{alignat}
    This verifies the three bounds in \eqref{eqn:three-numerical-theta-bounds} for all $N$ with $N \geq 584,\!000,\!000$. 

    Then via exhaustive computer check over all $N < 584,\!000,\!000$, we obtain the claimed bounds from the table. See \cite{ross-code} for the code. 
    % Note the $n$-th column in the table comes from the $\theta_i$ bounds from $P_{n-2}$.
\end{proof}

Next, we bound the inner summation for $A_{2,m}$ in \eqref{eqn:A2m-formula}. 
\begin{lemma} \label{lem:U-mu-bound}  
    For $t,n,m$ given in \eqref{eqn:A2m-formula}, 
    $$
    \lrabs{U_{k-1}(t,m) \cdot \mu(t,n,m) } \leq 2 \psi(n) 2^{\omega(N)} m^{(k-1)/2}.
    $$
\end{lemma}
\begin{proof}
Recall that $m \geq 1$, $t \in \ZZ$ such that $4m-t^2 > 0$, and $n \geq 1$ such that $n^2 \, | \, (t^2 -4m)$ and $\frac{t^2-4m}{n^2} \equiv 0,1 \, \pmod{4}$. And recall that $U_{k-1}(t,m) = \frac{\rho^{k-1} - \bar{\rho}^{k-1}}{\rho - \bar{\rho}}$ where $\rho, \bar{\rho}$ are the two roots of the polynomial $X^2-tX+m$. Finally, recall that $\displaystyle \mu (t,n,m) = \frac{\psi (N)}{\psi (N/N_n)} \sideset{}{'} \sum_{c \!\! \mod N} \chi(c)$, where $N_n = \gcd(N,n)$, and the primed summation runs through all elements $c$ of $(\ZZ / N \ZZ)^{\times}$ which lift to solutions of $c^2-tc+m \equiv 0 \, \pmod{NN_n}$.

So,
\begin{equation} \label{eqn:U-mu-factors}
\lrabs{U_{k-1}(t,m) \cdot \mu(t,n,m) } = \lrabs{ \frac{\rho^{k-1} - \bar{\rho}^{k-1}}{\rho - \bar{\rho}} } \cdot \frac{\psi(N)}{\psi(N/N_n)} \cdot \lrabs{  \sideset{}{'} \sum_{c \!\!\! \mod N} \chi(c) } .
\end{equation}
We give bounds on each of these three factors.

First, since 
$$\lrabs{\rho} = \sqrt{m} \qquad \text{and} \qquad \lrabs{\rho - \bar{\rho}} = \sqrt{4m-t^2} ,$$ 
we have
\begin{equation} \label{eqn:U-mu-bound1}
\lrabs{ \frac{\rho^{k-1} - \bar{\rho}^{k-1}}{\rho - \bar{\rho}} } \leq \frac{ \lrabs{\rho^{k-1}} + \lrabs{\bar{\rho}^{k-1} } }{ \lrabs{ \rho - \bar{\rho} } } = \frac{2 m^{(k-1)/2}}{\sqrt{4m-t^2}} .
\end{equation}

Second, note that for every prime $p \mid N$, we will either have $p \mid N_n$ or $p \mid N/N_n$. Thus
\begin{align}
    \psi(N) &= N \prod_{p|N} \lrp{1 + \frac{1}{p} } \\
    &\leq N_n \prod_{p|N_n} \lrp{1 + \frac{1}{p} } \cdot N/N_n \prod_{p|N/N_n} \lrp{1 + \frac{1}{p}} \\
    &= \psi(N_n) \psi(N/N_n) ,
\end{align}
which yields
\begin{equation} \label{eqn:U-mu-bound2}
    \frac{\psi(N)}{\psi(N/N_n)} \leq \psi(N_n) \leq \psi(n)  ,
\end{equation}
where the second inequality comes from the fact that $N_n \mid n$.

Third, note for every term $c$ in the sum of \eqref{eqn:U-mu-factors}, $c^2-tc+m \equiv 0 \pmod{NN_n}$ means that $c^2-tc+m \equiv 0 \pmod{N}$ as well. By {\cite[Lemma~2]{serre}}, the congruence $x^2-tx+m \equiv 0 \pmod{N}$ has at most $2^{\omega(N)} \sqrt{4m-t^2}$ solutions. Thus 
\begin{equation} \label{eqn:U-mu-bound3}
\lrabs{  \sideset{}{'} \sum_{c \!\!\! \mod N} \chi(c) } \leq 2^{\omega(N)} \sqrt{4m-t^2} .
\end{equation}
Combining the bounds \eqref{eqn:U-mu-bound1}, \eqref{eqn:U-mu-bound2}, and \eqref{eqn:U-mu-bound3}, we obtain
$$
    \lrabs{U_{k-1}(t,m) \cdot \mu(t,n,m) } \leq 2 \psi(n) 2^{\omega(N)} m^{(k-1)/2} ,
$$
which completes the proof.
\end{proof}

Next, we bound the inner summation for $A_{3,m}$ in \eqref{eqn:A3m-formula}. 
\begin{lemma} \label{lem:sigma-N-m-d-bound}
    Let 
$$ \Sigma(N,m,d) := \sum_\tau \phi(\gcd(\tau, N/\tau)) \chi(y_{\tau})
$$
denote the inner summation for $A_{3,m}$ in \eqref{eqn:A3m-formula}. 
Then 
$$\lrabs{\Sigma(N,m,d)} \leq
    \begin{cases} 
            \lrabs{d-\frac md} \cdot 2^{\omega(N)} & \text{if } d \neq \sqrt{m}, \\
            \sqrt{N} \cdot 2^{\omega(N)} & \text{in general.}
    \end{cases}
$$
\end{lemma}
\begin{proof}
Recall that the summation $\sum_{\tau}$  runs over all positive divisors $\tau$ of $N$ such that $\gcd(\tau,N/\tau)$ divides $\gcd(N/N_{\chi}, d-m/d)$. Additionally, $y_{\tau}$ is the unique integer modulo $\lcm(\tau, N/\tau)$ determined by the congruences $y_{\tau} \equiv d \pmod{\tau}$ and $y_{\tau} \equiv m/d \pmod{N/\tau}$.

    Let $h := |d - \frac{m}{d}|$. Then
    \begin{align}
        \lrabs{\Sigma(N,m,d)} &= \lrabs{\sum_{\substack{\tau \mid N \\  (\tau, N/\tau) \mid (h,N/N_{\chi})}} \phi(\gcd(\tau, N/\tau)) \chi(y_{\tau})} \\
        &\leq \sum_{\substack{\tau \mid N \\  (\tau, N/\tau) \mid (h,N/N_{\chi})}} \phi(\gcd(\tau, N/\tau)) \\
        &\leq \sum_{\substack{\tau \mid N \\  (\tau, N/\tau) \mid h}} \phi(\gcd(\tau, N/\tau)) \\
        &= \sum_{\delta\mid h} \ \sum_{\substack{\tau \mid N \\ (\tau, N/\tau)=\delta}} \phi(\delta) \\
        &= \sum_{\delta \mid h} \phi(\delta) \cdot \#\{\tau \mid N \ \colon \ \gcd(\tau, N/\tau)=\delta\} .
    \end{align}
    
    In the case of $d \neq \sqrt{m}$, we have $h\neq 0$, and so using Lemma \ref{lem:tau-set-cardinality-bound} below, we have
    \begin{align}
        \lrabs{\Sigma(N,m,d)} 
        &\leq \sum_{\delta \mid h} \phi(\delta) \cdot \#\{\tau|N \ \colon \ \gcd(\tau, N/\tau)=\delta\} \\
        &\leq \sum_{\delta \mid h} \phi(\delta) \cdot 2^{\omega(N)} \\
        &= h \cdot 2^{\omega(N)}.
    \end{align}
    Here, we used the well-known formula
    $\sum_{\delta \mid h} \phi(\delta) = h$.
    
    In the general case, write $N = DM^2$ where $D$ is squarefree. Then for any $\delta$ such that $\gcd(\tau, N/\tau) = \delta$ for some $\tau$, note $\delta \mid \tau$ and $\delta \mid N/\tau$ imply $\delta^2 \mid N$, which means that $\delta \mid M$. 
    Hence by Lemma \ref{lem:tau-set-cardinality-bound} again, 
    \begin{align}
        \lrabs{\Sigma(N,m,d)} 
        &\leq \sum_{\delta \mid h} \phi(\delta) \cdot \#\{\tau \mid N \ \colon \ (\tau, N/\tau)=\delta\} \\
        &\leq \sum_{\delta \mid M} \phi(\delta) \cdot \#\{\tau \mid N \ \colon \ (\tau, N/\tau)=\delta\} \\
        &\leq \sum_{\delta \mid M} \phi(\delta) \cdot 2^{\omega(N)} \\
        &= M \cdot 2^{\omega(N)} \\
        &\leq \sqrt{N} \cdot 2^{\omega(N)},
    \end{align}
    which completes the proof.
\end{proof}

\begin{lemma} \label{lem:tau-set-cardinality-bound}
    Let $N$ and $\delta$ be positive integers. Then
    $$\#\{\tau \mid N \ \colon \ \gcd(\tau, N/\tau)=\delta\} \leq 2^{\omega(N)} .$$
\end{lemma}
\begin{proof}
    % $\#\{\tau \mid N \ \colon \ \gcd(\tau, N/\tau)=\delta\} \leq 2^{\omega(N)}$. Factor $\delta = p_1^{r_1} \cdots p_s^{r_s}$, and 
    Without loss of generality, we can assume that $\delta \mid N$ (otherwise, the inequality holds trivially). 
    Consider the possible $\tau$ that would yield $\gcd(\tau, N/\tau)=\delta$.
    For each prime $p \mid N$,
    let $v_p(\cdot)$ denote $p$-adic valuation, and let $r_p := v_p(\delta)$.  Note that $\gcd(\tau, N/\tau)=\delta$ precisely for $\tau$ such that $v_p(\gcd(\tau,N/\tau)) = r_p$ for each $p \mid N$. On the other hand, $v_p( \gcd(\tau,N/\tau)) = r_p$ means that either $v_{p}(\tau) = r_p$ and $v_p(N/\tau) \geq r_p$, or $v_p(N/\tau) = r_p$ and $v_{p}(\tau) \geq r_p$. This yields only two possible values for $v_{p}(\tau)$: $r_p$ and $v_{p}(N) - r_p$. Thus, since there are at most two possible options for $v_{p}(\tau)$ for each prime $p \mid N$, we have that $\#\{\tau \mid N \ \colon \ \gcd(\tau, N/\tau)=\delta\} \leq  2^{\omega(N)}$. 
\end{proof}

\section{Proof of Theorems \ref{thm:main} and \ref{thm:main-trivial-char}} \label{sec:proofofmain}
In this section, we will show Theorems \ref{thm:main} and \ref{thm:main-trivial-char}. We split the proof into the case when $m$ is not a perfect square (Proposition \ref{prop:main-theorem-not-perfect-square}), and the case when $m$ is a perfect square (Proposition \ref{prop:main-theorem-perfect-square}).

Recall what big $O$ notation means in terms of two variables $N$ and $k$. A function $f(N,k)$ is $O(g(N,k))$ if there exists a constant $C$ such that $\lrabs{f(N,k)} \leq C \cdot g(N,k)$ for $N+k$ sufficiently large.
In other words, for any fixed value of $N$, this can be interpreted as big $O$ notation with respect to $k$, and for any fixed value of $k$, this can be interpreted as big $O$ notation with respect to $N$.

Now, we give a bound on the trace $\Tr T_m(N,k,\chi)$ when $m$ is not a perfect square. 
\begin{lemma} \label{lem:main-lemma-not-perfect-square}
    Let $m\ge1$ be fixed such that $m$ is not a perfect square. For all $N \geq 1$, $k \geq 2$, and $\chi$ a Dirichlet character modulo $N$ with $\chi(-1) = (-1)^k$, we have
    $$\Tr T_m(N,k,\chi) =  O(2^{\omega(N)} m^{k/2}) .$$
\end{lemma}
\begin{proof}
    We examine each of the $A_{i,m}$ terms in \eqref{eqn:eichler-selberg-trace-formula} separately.
    
    First, since $m$ is not a perfect square and $\chi(\sqrt{m})=0$, we have that $A_{1,m} = 0$ by \eqref{eqn:A1m-formula}.
    
    Second, observe that all the $t$ and $n$ from \eqref{eqn:A2m-formula} are bounded by the fixed value of $2\sqrt{m}$. Thus by Lemma \ref{lem:U-mu-bound},  
    $$\lrabs{U_{k-1}(t,m) \cdot \mu(t,n,m) } = O(2^{\omega(N)} m^{k/2}) ,$$ 
    so by \eqref{eqn:A2m-formula}
    \begin{align*}
        A_{2,m} &= \frac{1}{2} \sum_{t^2 < 4m} \sum_{n}  h_w \lrp{ \frac{t^2-4m}{n^2} }  U_{k-1}(t,m) \mu(t,n,m) \\
        &= O(2^{\omega(N)} m^{k/2}) .
    \end{align*}

    Third, since $m$ is not a perfect square, only the first case of Lemma \ref{lem:sigma-N-m-d-bound} applies. Thus each inner summation $\Sigma(N,m,d)$ for $A_{3,m}$ is $O(2^{\omega(N)})$. And also note that $\min(d,m/d) \leq m^{1/2}$. So by \eqref{eqn:A3m-formula},
    \begin{align*}
        A_{3,m} &= \frac{1}{2} \sum_{d|m} \min(d,m/d)^{k-1} \Sigma(N,m,d) \\
        &= O(2^{\omega(N)} m^{k/2}).
    \end{align*}
    Fourth, $A_{4,m} \leq \sum_{c\mid m} c = O(1)$ by \eqref{eqn:A4m-formula}. 

    Combining the above bounds, we obtain
    \begin{align*}
        \Tr T_m(N,k,\chi) &= A_{1,m} - A_{2,m} - A_{3,m} + A_{4,m} \\
        &=  O(2^{\omega(N)} m^{k/2}) ,
    \end{align*}
    which completes the proof.
\end{proof}

Next, we estimate the trace $\Tr T_m(N,k,\chi)$ when $m$ is a perfect square. 
\begin{lemma} \label{lem:main-lemma-perfect-square}
    Let $m$ be fixed such that $m$ is a perfect square. Then for all $N \geq 1$, $k \geq 2$, and $\chi$ a Dirichlet character modulo $N$ with $\chi(-1) = (-1)^k$, we have
    $$\Tr T_m(N,k,\chi) = \chi(\sqrt{m}) \frac{k-1}{12} \psi(N) m^{k/2-1} + O(\sqrt{N} 2^{\omega(N)} m^{k/2}).$$
\end{lemma}
\begin{proof}
    We examine each of the $A_{i,m}$ terms in \eqref{eqn:eichler-selberg-trace-formula} separately.
    
    First, $A_{1,m} = \chi(\sqrt{m}) \frac{k-1}{12} \psi(N) m^{k/2-1}$ from \eqref{eqn:A1m-formula}. 
    
    Second, like in the previous Lemma, $A_{2,m} = O(2^{\omega(N)} m^{k/2})$.

    Third, from Lemma \ref{lem:sigma-N-m-d-bound}, each inner summation $\Sigma(N,m,d)$ in $A_{3,m}$ is $O(\sqrt{N} 2^{\omega(N)})$. So by \eqref{eqn:A3m-formula} and the fact that $\min(d,m/d) \leq m^{1/2}$,
    \begin{align*}
        A_{3,m} &= \frac{1}{2} \sum_{d|m} \min(d,m/d)^{k-1} \Sigma(N,m,d) \\
        &= O(\sqrt{N} 2^{\omega(N)} m^{k/2}).
    \end{align*}
    
    Fourth, we have $A_{4,m} = O(1)$ by \eqref{eqn:A4m-formula}.

    Combining the above bounds, we obtain
    \begin{align*}
        \Tr T_m(N,k,\chi) &= A_{1,m} - A_{2,m} - A_{3,m} + A_{4,m} \\
        &= \chi(\sqrt{m}) \frac{k-1}{12} \psi(N) m^{k/2-1} + O(\sqrt{N} 2^{\omega(N)} m^{k/2}),
    \end{align*}
    concluding the proof.
\end{proof}

Next, we prove Theorems \ref{thm:main} and \ref{thm:main-trivial-char} in the case when $m$ is not a perfect square.
\begin{proposition} \label{prop:main-theorem-not-perfect-square}
    Let $m \geq 1$ be fixed and not a perfect square. Suppose that $\gcd(N,m)=1$, $k \geq 2$, and $\chi$ is a Dirichlet character modulo $N$ with $\chi(-1) = (-1)^k$. Then  $a_2(m,N,k,\chi)$ is nonvanishing for all but finitely many triples $(N,k,\chi)$. 
    Furthermore, when $\chi$ is trivial, 
    $a_2(m,N,k)$ is negative for all but finitely many pairs $(N,k)$.
\end{proposition}
\begin{proof}
    By Lemma \ref{lem:a2-coeff-formula}, we have
    \begin{align*}
        a_2(m,N,k,\chi) 
        &= \frac{1}{2} \lrb{ (\Tr T_m)^2 - \sum_{d \mid m} \chi(d) d^{k-1} \Tr T_{m^2/d^2} }.  
    \end{align*}
    Now, observe that every term inside the sum has $m^2/d^2$ a perfect square. Thus for each term in the sum, we have by Lemma \ref{lem:main-lemma-perfect-square}, 
        \begin{align*}
        &~~\chi(d) d^{k-1} \Tr T_{m^2/d^2} \\
        &=\chi(d) d^{k-1} 
        \lrb{
        \chi \lrp{\sqrt{\frac{m^2}{d^2}}} \frac{k-1}{12} \psi(N) \lrp{\frac{m^2}{d^2}}^{k/2-1} + O\left( \sqrt{N} 2^{\omega(N)} \lrp{\frac{m^2}{d^2}}^{k/2} \right)  
        } \\
        &=\chi(d) d^{k-1} 
        \lrb{
        \chi \lrp{\frac{m}{d}} \frac{k-1}{12} \psi(N) \frac{m^{k-2}}{d^{k-2}} + O\left( \sqrt{N} 2^{\omega(N)} \frac{m^k}{d^k} \right)  
        } \\
        &=\chi(m) d \,\frac{k-1}{12} \psi(N) m^{k-2} + O(\sqrt{N} 2^{\omega(N)} m^k) .
    \end{align*}
    This yields
    \begin{align*}
        a_2(m,N,k,\chi) &= \frac{1}{2} \lrb{ (\Tr T_m)^2  - \sum_{d \mid m} \chi(d) d^{k-1} \Tr T_{m^2/d^2} } \\
        &= \frac{1}{2} \lrb{ (\Tr T_m)^2  - \sum_{d \mid m} \lrb{
        \chi(m) d \,\frac{k-1}{12} \psi(N) m^{k-2} + O(\sqrt{N} 2^{\omega(N)} m^k)
        } } \\
        &= \frac{1}{2} \lrb{ (\Tr T_m)^2  - 
        \chi(m) \frac{k-1}{12} \psi(N) m^{k-2} \sigma_1(m) + O(\sqrt{N} 2^{\omega(N)} m^k)
        } .
    \end{align*}
    Since $m$ is not a perfect square, we apply Lemma \ref{lem:main-lemma-not-perfect-square} to $\Tr T_m$ and obtain
    \begin{align}
        a_2(m,N,k,\chi)
        &= \frac{1}{2} \lrb{ O(2^{\omega(N)} m^{k/2})^2  - 
        \chi(m) \frac{k-1}{12} \psi(N) m^{k-2} \sigma_1(m) - O(\sqrt{N} 2^{\omega(N)} m^k)
        } \\
        &= \frac{1}{2} \lrb{ - 
        \chi(m) \frac{k-1}{12} \psi(N) m^{k-2} \sigma_1(m) + O(\sqrt{N} 2^{\omega(N)} m^k)
        } \\
        &= \frac{\psi(N) m^{k-2} \sigma_1(m)}{2} \lrb{ 
        - \chi(m) \frac{k-1}{12}  
         + O(\theta_1(N))
        } . \label{eqn:main-theorem-not-perfect-square-final-eqn}
    \end{align}
    Here we used the fact that $2^{\omega(N)} = O(\sqrt{N})$ (from the proof of Lemma \ref{lem:omega-over-psi-bounds}) and the definition $\theta_1(N) = \frac{\sqrt{N} 2^{\omega(N)}}{\psi(N)}$ (also from Lemma \ref{lem:omega-over-psi-bounds}).

    Now, $\gcd(m,N)=1$ means that $\lrabs{\chi(m)} = 1$, and hence  $\lrabs{\chi(m) \frac{k-1}{12}} \geq \frac{1}{12}$ for all $k \geq 2$. But for sufficiently large $N$ (independent of $k$), the $O(\theta_1(N))$ term will be $< \frac{1}{12}$ in magnitude since $\theta_1(N) \longrightarrow 0$ according to Lemma \ref{lem:omega-over-psi-bounds}. Thus $a_2(m,N,k,\chi)$  will be nonvanishing for sufficiently large $N$. 

    This leaves only finitely many $N$ to check. For these values of $N$, the $O(\theta_1(N))$ term will be bounded by a constant. So for $k$ sufficiently large, $\lrabs{-\chi(m) \frac{k-1}{12}} = \frac{k-1}{12}$ will be larger than that constant, and $a_2(m,N,k,\chi)$ will be nonvanishing.

    This shows that $a_2(m,N,k,\chi)$ is nonvanishing for all but finitely many triples $(N,k,\chi)$.

    Furthermore, note that when $\chi$ is trivial, $a_2(m,N,k)$ will be real. In this case, the expression inside the brackets of \eqref{eqn:main-theorem-not-perfect-square-final-eqn} becomes $- \frac{k-1}{12} + O(\theta_1(N))$. So in particular, $a_2(m,N,k)$ will be negative for all but finitely many pairs $(N,k)$.
\end{proof}

Finally, we prove Theorems \ref{thm:main} and \ref{thm:main-trivial-char} in the case when $m$ is a perfect square.

\begin{proposition} \label{prop:main-theorem-perfect-square}
    Let $m \geq 1$ be fixed and a perfect square. Suppose that $\gcd(N,m)=1$, $k \geq 2$, and $\chi$ is a Dirichlet character modulo $N$ with $\chi(-1) = (-1)^k$. Then $a_2(m,N,k,\chi)$ is nonvanishing for all but finitely many triples $(N,k,\chi)$. 
    Furthermore, when $\chi$ is trivial, 
    $a_2(m,N,k)$ is positive for all but finitely many pairs $(N,k)$.
\end{proposition}
\begin{proof}
    By the same argument as in Proposition \ref{prop:main-theorem-not-perfect-square}, 
    \begin{align*}
        a_2(m,N,k,\chi) &= \frac{1}{2} \lrb{ (\Tr T_m)^2  - 
        \chi(m) \frac{k-1}{12} \psi(N) m^{k-2} \sigma_1(m) + O(\sqrt{N} 2^{\omega(N)} m^k) } .
    \end{align*}
    Then using the fact that $\sqrt{N} 2^{\omega(N)} = O\lrp{\psi(N)}$ from Lemma \ref{lem:omega-over-psi-bounds}, this yields
    \begin{align}
        a_2(m,N,k,\chi) &= \frac{1}{2} \lrb{ (\Tr T_m)^2  + O(k \psi(N) m^k) }. 
    \end{align}

    Since $m$ is a perfect square, we apply Lemma \ref{lem:main-lemma-perfect-square} to $\Tr T_m$, and we have
    \begin{align*}
        (\Tr T_m)^2 &= \lrp{  \chi(\sqrt{m}) \frac{k-1}{12} \psi(N) m^{k/2-1} + O(\sqrt{N} 2^{\omega(N)} m^{k/2})  }^2  \\
        &= \chi(m) \frac{(k-1)^2}{144} \psi(N)^2 m^{k-2} + O\lrp{k \psi(N) m^{k/2} \cdot  \sqrt{N} 2^{\omega(N)} m^{k/2}} + O\lrp{\sqrt{N} 2^{\omega(N)} m^{k/2}}^2 \\
        &= \chi(m) \frac{(k-1)^2}{144} \psi(N)^2 m^{k-2} + O\lrp{k \psi(N)  \sqrt{N} 2^{\omega(N)} m^k} .
    \end{align*}
Here, we again used the fact that $\sqrt{N} 2^{\omega(N)} = O\lrp{\psi(N)}$.
Then
    \begin{align}
        a_2(m,N,k,\chi) &= \frac{1}{2} \lrb{ (\Tr T_m)^2  + O(k \psi(N) m^k) } \\
        &= \frac{1}{2} \lrb{ \chi(m) \frac{(k-1)^2}{144} \psi(N)^2 m^{k-2} + O\lrp{k \psi(N)  \sqrt{N} 2^{\omega(N)} m^k} 
        }  \\ 
        &= \frac{k \psi(N)^2 m^{k-2}}{2} \lrb{ 
          \chi(m) \frac{(k-1)^2}{144k} + O(\theta_1(N))  
        }. \label{eqn:main-theorem-perfect-square-final-eqn}
    \end{align}

    Now, $\lrabs{\chi(m) \frac{(k-1)^2}{144k}} \geq \frac{1}{288}$ for all $k \geq 2$. But for sufficiently large $N$, the $O(\theta_1(N))$ term will be $< \frac{1}{288}$ in magnitude, according to Lemma \ref{lem:omega-over-psi-bounds}. Thus $a_2(m,N,k,\chi)$  will be nonvanishing for sufficiently large $N$.

    Again, this leaves only finitely many $N$ to check. For these values of $N$, the $O(\theta_1(N))$ term will be bounded by a constant, by Lemma \ref{lem:omega-over-psi-bounds}. So for $k$ sufficiently large, $\lrabs{\chi(m) \frac{(k-1)^2}{144k}} = \frac{(k-1)^2}{144k}$ will be larger than that constant, and $a_2(m,N,k,\chi)$ will be nonvanishing.

    This shows that $a_2(m,N,k,\chi)$ is nonvanishing for all but finitely many triples $(N,k,\chi)$.

    Furthermore, note that when $\chi$ is trivial, the expression inside the brackets of \eqref{eqn:main-theorem-perfect-square-final-eqn} becomes $\frac{(k-1)^2}{144k} + O(\theta_1(N))$. So in particular, $a_2(m,N,k)$ will be positive for all but finitely many pairs $(N,k)$.
\end{proof}

Propositions \ref{prop:main-theorem-not-perfect-square} and \ref{prop:main-theorem-perfect-square} combine to imply Theorems \ref{thm:main} and \ref{thm:main-trivial-char}. 

In these two proofs, we needed $\gcd(m,N)=1$ to use the fact that $\lrabs{\chi(m)} = 1$. 
It is worth noting that this is the first place that we need this coprimality assumption. In particular, all of the trace bounds given above still work even when $\gcd(m,N) > 1$. 

Note that all of the bounds given here are explicitly computable. So in Section \ref{sec:proofofT3}, we compute all of these bounds for $m=3$, and give the complete list of pairs $(N,k)$ for which $a_2(3,N,k)$ is positive or zero. In Section \ref{sec:proofofT4}, we compute the bounds for $m=4$, and give the complete list of pairs $(N,k)$ for which $a_2(4,N,k)$ is negative or zero.

\section{Proof of Theorem \ref{thm:T3} } \label{sec:proofofT3}
In this section, we show Theorem \ref{thm:T3}: for $N$ coprime to $3$ and $k \geq 2$ even, $a_2(3,N,k)$ is positive or zero only for the pairs $(N,k)$ given in Table \ref{table:T3-pairs-dim-a2}.

Observe that in the notation of Lemma \ref{lem:a2-coeff-formula}, for $\gcd(N,3)=1$ and $k \geq 2$ even, we have
\begin{equation} \label{eqn:a2-coeff-T3}
    a_2(3,N,k) = \frac{1}{2} \lrb{ (\Tr T_3)^2 - \Tr T_9 -3^{k-1} \Tr T_1 }.
\end{equation}

We bound the terms of \eqref{eqn:a2-coeff-T3} separately in the following three lemmas.
Each of these bounds will be expressed in terms of the $\theta_i(N)$ defined in Lemma \ref{lem:omega-over-psi-bounds}.

\begin{lemma} \label{lem:trace-T3-bound}
    Let $N \geq 1$ and $k \geq 2$ be even. Then
    \begin{align*}
        \frac{(\Tr T_3)^2}{\psi(N) 3^k } 
        &\leq  \frac{448 + 160\sqrt{3}}{27} \theta_2(N).   
    \end{align*}
\end{lemma}
\begin{proof}
We examine each of the $A_{i,3}$ terms from \eqref{eqn:eichler-selberg-trace-formula} separately.

First, by \eqref{eqn:A1m-formula},
$$A_{1,3} = \chi_0(\sqrt{3}) \frac{k-1}{12} \psi (N) 3^{k/2-1} = 0.$$

Second, by \eqref{eqn:A2m-formula} (recalling that the $t=t_0$ and $t=-t_0$ terms coincide) and Table \ref{table:weighted-class-numbers},
\begin{align}
    A_{2,3} = &\frac{1}{2} \sum_{t^2 < 12} U_{k-1}(t,3) \sum_{n}  h_w \lrp{ \frac{t^2-12}{n^2} } \mu (t,n,3) \\
    = &\frac{1}{2} U_{k-1}(0,3) \sum_{n}  h_w \lrp{ \frac{-12}{n^2} } \mu(0,n,3) \\
    &+ U_{k-1}(1,3)  \sum_{n}  h_w \lrp{ \frac{-11}{n^2} } \mu (1,n,3) \\
    &+ U_{k-1}(2,3) \sum_{n}  h_w \lrp{ \frac{-8}{n^2} } \mu (2,n,3) \\
    &+ U_{k-1}(3,3) \sum_{n}  h_w \lrp{ \frac{-3}{n^2} } \mu (3,n,3) \\
    = &\frac{1}{2} U_{k-1}(0,3) \lrb{ h_w(-12) \mu(0,1,3) + h_w(-3) \mu(0,2,3) } \\
    &+ U_{k-1}(1,3) \lrb{ h_w(-11) \mu (1,1,3) } \\
    &+ U_{k-1}(2,3) \lrb{ h_w(-8) \mu(2,1,3) } \\
    &+ U_{k-1}(3,3) \lrb{ h_w(-3) \mu(3,1,3) } \\
    = &\frac{1}{2} U_{k-1}(0,3) [ \mu(0,1,3) + \frac{1}{3} \mu(0,2,3) ] \\
    &+U_{k-1}(1,3) [ \mu (1,1,3) ] \\
    &+U_{k-1}(2,3) [ \mu(2,1,3) ] \\
    &+U_{k-1}(3,3) [ \frac{1}{3} \mu(3,1,3) ]. 
\end{align}
So by Lemma \ref{lem:U-mu-bound},
\begin{align}
    \lrabs{A_{2,3}} &\leq  \, 2^{\omega(N)} 3^{(k-1)/2}\cdot 2 \lrb{\frac12 \lrp{1 + \frac{1}{3} \cdot 3} + (1) + (1) + \lrp{\frac13 } } \\
    &= \frac{20}{3 \sqrt{3}} 2^{\omega(N)} 3^{k/2}.
\end{align}

Third, by \eqref{eqn:A3m-formula} (recalling that the $d=d_0$ and $d=m/d_0$ terms coincide) and Lemma \ref{lem:sigma-N-m-d-bound}, 
\begin{align}
    A_{3,3} &= \frac{1}{2} \sum_{d|3} \min(d,3/d)^{k-1} \Sigma(N,3,d) \\
    &= \Sigma(N,3,1)  \\
    &\leq  |-2| \cdot 2^{\omega(N)} \\
    &= 2\cdot 2^{\omega(N)}.
\end{align}

Fourth, by \eqref{eqn:A4m-formula},
\begin{align}
   A_{4,3}  &\leq \sum_{c\mid 3} c \ =  1+3=4. 
\end{align}

Now, note that $A_{3,3}$ and $A_{4,3}$ here are both positive and $\leq 4 \cdot 2^{\omega(N)}$. So putting all this together, we obtain 
% \leq \frac{4}{3} 3^{k/2} \, 2^{\omega(N)}
\begin{align} 
    |\Tr T_3 | &= |A_{1,3} - A_{2,3} - A_{3,3} + A_{4,3} | \\
    &\leq |A_{2,3}| + |A_{3,3} - A_{4,3} | \\
    &\leq \frac{20}{3\sqrt{3}} \cdot 2^{\omega(N)} 3^{k/2} + 4 \cdot 2^{\omega(N)} \\
    &\leq \frac{20}{3\sqrt{3}} \cdot 2^{\omega(N)} 3^{k/2} + \frac{4}{3} \cdot 3^{k/2}\,2^{\omega(N)} \\
    &= \lrp{\frac{20}{3\sqrt{3}} + \frac43}  2^{\omega(N)} 3^{k/2}  ,
\end{align}
which yields
\begin{align}
    \frac{(\Tr T_3)^2}{\psi(N) 3^k } 
    &\leq   \frac{\lrp{\frac{20}{3\sqrt{3}} + \frac43}^2  (2^{\omega(N)})^2 3^k}{\psi(N) 3^k} \\
    &= \frac{448 + 160\sqrt{3}}{27} \theta_2(N) ,
\end{align}
as desired.
\end{proof}

Next, we bound the $\Tr T_9$ term from \eqref{eqn:a2-coeff-T3}.

\begin{lemma} \label{lem:trace-T9-bound}
Let $N \geq 1$ with $\gcd(N,3)=1$ and $k \geq 2$ be even. Then
\begin{align}  
    \lrabs{\frac{\Tr T_9 - A_{1,9}}{\psi(N) 3^k}} 
    &\leq \frac{65}{6} \theta_3(N) + \frac{1}{6} \theta_1(N) ,
\end{align}
where
$$A_{1,9} = \frac{k-1}{108} \psi(N) 3^k .$$
\end{lemma}
\begin{proof}
We examine each of the $A_{i,9}$ terms from \eqref{eqn:eichler-selberg-trace-formula} separately.

First, by \eqref{eqn:A1m-formula},
\begin{align}
 A_{1,9} &= \chi_0(\sqrt{9}) \frac{k-1}{12} \psi (N) 9^{k/2-1} = \frac{k-1}{108} \psi (N) 3^k ,
\end{align}
as claimed.

Second, similarly to the proof of Lemma \ref{lem:trace-T3-bound}, we compute
\begin{align}
    A_{2,9} = &\frac{1}{2} \sum_{t^2 < 36} U_{k-1}(t,9) \sum_{n}  h_w \lrp{ \frac{t^2-36}{n^2} } \mu (t,n,9) \\
    = &\frac{1}{2} U_{k-1}(0,9) \sum_{n}  h_w \lrp{ \frac{-36}{n^2} } \mu(0,n,9) \\
     &+U_{k-1}(1,9) \sum_{n}  h_w \lrp{ \frac{-35}{n^2} } \mu(1,n,9) \\
     &+U_{k-1}(2,9) \sum_{n}  h_w \lrp{ \frac{-32}{n^2} } \mu(2,n,9) \\
     &+U_{k-1}(3,9) \sum_{n}  h_w \lrp{ \frac{-27}{n^2} } \mu(3,n,9) \\
     &+U_{k-1}(4,9) \sum_{n}  h_w \lrp{ \frac{-20}{n^2} } \mu(4,n,9) \\
     &+U_{k-1}(5,9) \sum_{n}  h_w \lrp{ \frac{-11}{n^2} } \mu(5,n,9) \\
     = &\frac{1}{2} U_{k-1}(0,9) \lrb{ h_w(-36)\mu(0,1,9) + h_w(-4)\mu(0,3,9) } \\
     &+U_{k-1}(1,9) \lrb{ h_w(-35)\mu(1,1,9) } \\
     &+U_{k-1}(2,9) \lrb{ h_w(-32) \mu(2,1,9) + h_w(-8)\mu(2,2,9) } \\
     &+U_{k-1}(3,9) \lrb{ h_w(-27)\mu(3,1,9) + h_w(-3)\mu(3,3,9) } \\
     &+U_{k-1}(4,9) \lrb{ h_w(-20) \mu(4,1,9) } \\
     &+U_{k-1}(5,9) \lrb{ h_w(-11) \mu(5,1,9) } \\
     = &\frac{1}{2} U_{k-1}(0,9) [ 2\mu(0,1,9) + \frac{1}{2} \mu(0,3,9) ] \\
     &+U_{k-1}(1,9) \lrb{ 2\mu(1,1,9) } \\
     &+U_{k-1}(2,9) \lrb{ 2\mu(2,1,9) + \mu(2,2,9) } \\
     &+U_{k-1}(3,9) [ \mu(3,1,9) + \frac{1}{3} \mu(3,3,9) ] \\
     &+U_{k-1}(4,9) \lrb{ 2\mu(4,1,9) } \\
     &+U_{k-1}(5,9) \lrb{ \mu(5,1,9) }.
\end{align}
So by Lemma \ref{lem:U-mu-bound},
\begin{align}
     |A_{2,9}| \leq & \, 2^{\omega(N)} 9^{(k-1)/2} \cdot 2 \lrb{\frac12 \left(2+\frac{1}{2}\cdot 4\right) + (2) + (2 + 3) + \left(1 + \frac13\cdot 4 \right) + (2) + (1) } \\
     =&  \frac{86}{9}\cdot 2^{\omega(N)} 3^k .
\end{align}
Third, by \eqref{eqn:A3m-formula} and Lemma \ref{lem:sigma-N-m-d-bound},
\begin{align}
    A_{3,9} &= \frac{1}{2} \sum_{d|9} \min(d,9/d)^{k-1} \Sigma(N,9,d) \\
    &=  \Sigma(N,9,1) + \frac{1}{2} \cdot3^{k-1} \Sigma(N,9,3)  \\
    &\leq  |-8|\cdot  2^{\omega(N)} + \frac{1}{2}\cdot 3^{k-1} \sqrt{N} 2^{\omega(N)} \\
    &= \frac{1}{6}\cdot 3^k \sqrt{N} 2^{\omega(N)} + 8 \cdot 2^{\omega(N)} .
\end{align}
Fourth, by \eqref{eqn:A4m-formula},
\begin{align}
   A_{4,9}  &\leq \sum_{c|9} c \ =  1+3+9=13 .
\end{align}
Now, note that $A_{3,9}$ and $A_{4,9}$ here are both positive and $\leq \frac{1}{6} \cdot 3^k \sqrt{N} 2^{\omega(N)} + \frac{23}{2} \cdot 2^{\omega(N)}$.
So putting this all together, we obtain
\begin{align}
    \lrabs{\Tr T_9 - A_{1,9}}
    &= \lrabs{- A_{2,9} - A_{3,9} + A_{4,9}} \\
    &\leq \lrabs{A_{2,9}} + \lrabs{A_{3,9} - A_{4,9}} \\
    &\leq \frac{86}{9}\cdot 2^{\omega(N)} 3^k + \frac{1}{6}\cdot 3^k \sqrt{N} 2^{\omega(N)} + \frac{23}{2}\cdot 2^{\omega(N)} \\
    &\leq \frac{86}{9}\cdot 2^{\omega(N)} 3^k + \frac{1}{6}\cdot 3^k \sqrt{N} 2^{\omega(N)} + \frac{23}{18}\cdot 3^k 2^{\omega(N)} \\
    &= \frac{65}{6}\cdot 2^{\omega(N)} 3^k + \frac{1}{6}\cdot 3^k \sqrt{N} 2^{\omega(N)} , 
\end{align}
which yields
\begin{align}
    \lrabs{\frac{\Tr T_9 - A_{1,9}}{\psi(N) 3^k}} 
    &\leq \frac{65}{6} \theta_3(N) + \frac{1}{6} \theta_1(N),
\end{align}
as desired.
\end{proof}

Finally, we bound the $\Tr T_1$ term from \eqref{eqn:a2-coeff-T3}.

\begin{lemma} \label{lem:trace-T1-bound1}
    Let $N \geq 1$ and $k \geq 2$ be even. Then
    $$\lrabs{\frac{\Tr T_1 - A_{1,1}}{\psi(N)}} \leq \frac{5}{3} \theta_3(N) + \frac{1}{2} \theta_1(N), $$
    where 
    $$A_{1,1} = \frac{k-1}{12} \psi(N) .$$
\end{lemma}
\begin{proof}
    We examine each of the $A_{i,1}$ terms from \eqref{eqn:eichler-selberg-trace-formula} separately.
    
    First, by \eqref{eqn:A1m-formula},   
    \begin{align}
        A_{1,1} &= \chi_0(\sqrt{1}) \frac{k-1}{12} \psi (N) 1^{k/2-1} = \frac{k-1}{12} \psi (N) ,
    \end{align}
    as claimed.
    
    Second, by \eqref{eqn:A2m-formula},
    \begin{align}
        A_{2,1} =& \frac{1}{2} \sum_{t^2 < 4} U_{k-1}(t,1) \sum_{n}  h_w \lrp{ \frac{t^2-4}{n^2} } \mu (t,n,1) \\
        =& \frac{1}{2} U_{k-1}(0,1) \sum_{n}  h_w \lrp{ \frac{-4}{n^2} } \mu(0,n,1) 
         + U_{k-1}(1,1) \sum_{n}  h_w \lrp{ \frac{-3}{n^2} } \mu(1,n,1) \\
        =& \frac{1}{2} U_{k-1}(0,1) h_w \lrp{ -4 } \mu(0,1,1)  + U_{k-1}(1,1) h_w \lrp{ -3 } \mu(1,1,1) \\
        =& \frac{1}{2} \cdot \frac{1}{2} U_{k-1}(0,1) \mu(0,1,1)  + \frac{1}{3} U_{k-1}(1,1) \mu(1,1,1).
    \end{align}
    So by Lemma \ref{lem:U-mu-bound},
    \begin{align}
        \lrabs{A_{2,1}} \leq \frac{1}{4} \cdot 2 \cdot 2^{\omega(N)} + \frac{1}{3} \cdot 2 \cdot 2^{\omega(N)} = \frac{7}{6}\cdot 2^{\omega(N)} .
    \end{align}
    Third, by \eqref{eqn:A3m-formula} and Lemma \ref{lem:sigma-N-m-d-bound},
    \begin{align}
        A_{3,1} &= \frac{1}{2} \sum_{d|1} \min(d,1/d)^{k-1} \Sigma(N,1,d) \\
        &=  \frac{1}{2} \Sigma(N,1,1)  \\
        &\leq \frac{1}{2} \sqrt{N} 2^{\omega(N)} .
    \end{align}
    Fourth, by \eqref{eqn:A4m-formula},
    $$A_{4,1} \leq \sum_{c|1} c \ = 1 .$$

    Now, note that $A_{3,1}$ and $A_{4,1}$ are both positive and $\leq \frac{1}{2} \sqrt{N} 2^{\omega(N)} + \frac{1}{2} \cdot 2^{\omega(N)}$. So putting all this together, we obtain
    \begin{align}
        \lrabs{\Tr T_1 - A_{1,1}}  
        &= \lrabs{ -A_{2,1} - A_{3,1} + A_{4,1} }  \\
        &\leq  \lrabs{A_{2,1}} + \lrabs{A_{3,1} - A_{4,1}} \\
        &\leq  \frac{7}{6}\cdot 2^{\omega(N)} + \frac{1}{2} \sqrt{N} 2^{\omega(N)} + \frac{1}{2}\cdot 2^{\omega(N)} \\
        &=  \frac{5}{3}\cdot 2^{\omega(N)} + \frac{1}{2} \sqrt{N} 2^{\omega(N)} ,
    \end{align}
    which yields
    \begin{align}
        \lrabs{\frac{\Tr T_1 - A_{1,1}}{\psi(N)}} &\leq \frac{5}{3} \theta_3(N) + \frac{1}{2} \theta_1(N), 
    \end{align}
    completing the proof.   
\end{proof}
Note that $\Tr T_1$ is just the dimension of $S_k(\Gamma_0(N))$. So Lemma \ref{lem:trace-T1-bound1} is just a dimension estimate. And in fact, one can also derive the bounds in this lemma from the dimension formula given in {\cite[p.~264]{cohen-stromberg}}.

We are now ready to prove Theorem \ref{thm:T3}.

{
\renewcommand{\thetheorem}{\ref{thm:T3}}
\begin{theorem}
    Suppose that $\gcd(N,3)=1$ and that $k \geq 2$ is even.
    Then the second coefficient $a_2(3,N,k)$ is positive or zero only for the pairs $(N,k)$ given in Table \ref{table:T3-pairs-dim-a2}.
\end{theorem}
\addtocounter{theorem}{-1}
}

\begin{proof} 
    By \eqref{eqn:a2-coeff-T3} and Lemmas \ref{lem:trace-T9-bound} and \ref{lem:trace-T1-bound1},
    \begin{align}
        &\ \ \ \  a_2(3,N,k)\\
        &= \frac{1}{2} \lrb{ (\Tr T_3)^2 - \Tr T_9 - 3^{k-1} \Tr T_1 } \\ 
        &= \frac{1}{2} \lrb{ (\Tr T_3)^2 - A_{1,9} - \lrp{\Tr T_9 - A_{1,9}} - 3^{k-1} A_{1,1} - 3^{k-1} \lrp{\Tr T_1 - A_{1,1} } } \\
        &= \frac{1}{2} \lrb{- \frac{k-1}{108} \psi(N) 3^k - \frac{k-1}{12} \psi(N) 3^{k-1} + (\Tr T_3)^2 - \lrp{\Tr T_9 - A_{1,9}} - 3^{k-1} \lrp{\Tr T_1 - A_{1,1} } } \\
        &= \frac{\psi(N) 3^k}{2} \lrb{ - \frac{k-1}{27} + \frac{(\Tr T_3)^2}{\psi(N) 3^k } - \lrp{ \frac{\Tr T_9 - A_{1,9}}{\psi(N) 3^k}  } - \frac{1}{3} \lrp{ \frac{\Tr T_1 - A_{1,1}}{\psi(N)}  } } \\
        &= \frac{\psi(N) 3^k}{2} \lrb{- \frac{k-1}{27} + E(N,k) } , \label{eqn:a2-3-final-eqn}
    \end{align}
    where $E(N,k)$ denotes the three error terms above. In particular, from Lemmas \ref{lem:trace-T3-bound}, \ref{lem:trace-T9-bound}, and \ref{lem:trace-T1-bound1},
    \begin{align}
        \lrabs{E(N,k)} &= \lrabs{  \frac{(\Tr T_3)^2}{\psi(N) 3^k } - \lrp{ \frac{\Tr T_9 - A_{1,9}}{\psi(N) 3^k}  } - \frac{1}{3} \lrp{ \frac{\Tr T_1 - A_{1,1}}{\psi(N)} }  }   \\
        &\leq  \frac{448 + 160\sqrt{3}}{27} \theta_2(N)  +  \lrp{\frac{65}{6} \theta_3(N) + \frac{1}{6} \theta_1(N)} + \frac{1}{3} \lrp{\frac{5}{3} \theta_3(N) + \frac{1}{2} \theta_1(N)} \\
        &=  \frac{448 + 160\sqrt{3}}{27} \theta_2(N)  +  \frac{205}{18} \theta_3(N) + \frac{1}{3} \theta_1(N) .  
    \end{align}
    
    For $N \geq 63,\!000,\!000$, we have the bounds  $\theta_1(N) \leq 0.0106$, $\theta_2(N) \leq 0.000314$, and $\theta_3(N) \leq 0.000015$ given in Lemma \ref{lem:omega-over-psi-bounds}. 
    Thus 
    \begin{align}
        \lrabs{E(N,k)} &\leq \frac{448 + 160\sqrt{3}}{27} 0.000314  +  \frac{205}{18} 0.000015 + \frac{1}{3} 0.0106 \\
        &\leq 0.0122 ,
    \end{align}
    which is $< \frac{k-1}{27}$ for all $k \geq 2$. By \eqref{eqn:a2-3-final-eqn}, this shows that $a_2(3,N,k) < 0$ for $N \geq 63,\!000,\!000$ and $k \geq 2$. 

    Utilizing the table in Lemma \ref{lem:omega-over-psi-bounds}, an identical argument using $2,\!700,\!000$, $150,\!000$, $8,\!800$, $571$, $43$, and $1$ as the bounds for $N$, shows that $a_2(3,N,k) < 0$ for $N \geq 2,\!700,\!000, k \geq 4$; $N \geq 150,\!000, k \geq 10$; $N \geq 8,\!800,\ k \geq 34$; $N \geq 571, k \geq 116$; $N \geq 43, k \geq 346$; and $N \geq 1, k \geq 1290$. 

    We then check the finite number of cases left by computer, which yields the complete list given in Table \ref{table:T3-pairs-dim-a2}. See \cite{ross-code} for the code.
\end{proof}

\begin{mytable}
\label{table:T3-pairs-dim-a2}
\begin{equation}
\setlength{\arraycolsep}{2.5mm}
\begin{array}{|c|c|c||c|c|c||c|c|c||c|c|c|} 
 \hline
 \multicolumn{12}{|c|}{\makecell{
    \text{All pairs $(N,k)$ for which $a_2(3,N,k)$ is positive or zero,} \\  
    \text{ along with $\dim S_k(\Gamma_0(N))$ and the actual value of $a_2(3,N,k)$.}  
  }} \\
 \hline
 (N,k) & \dim & a_2 & (N,k) & \dim & a_2 & (N,k) & \dim & a_2 & (N,k) & \dim & a_2 \\
 \hline
 \hline
 (1,2) & 0 & 0 & (8,2) & 0 & 0 & (4,6) & 1 & 0 & (4,8) & 2 & 144 \\
 \hline
 (1,4) & 0 & 0 & (10,2) & 0 & 0 & (5,4) & 1 & 0 & (22,2) & 2 & 1 \\
 \hline
 (1,6) & 0 & 0 & (13,2) & 0 & 0 & (5,6) & 1 & 0 & (28,2) & 2 & 4 \\
 \hline
 (1,8) & 0 & 0 & (16,2) & 0 & 0 & (7,4) & 1 & 0 & (34,2) & 3 & 0 \\
 \hline
 (1,10) & 0 & 0 & (25,2) & 0 & 0 & (8,4) & 1 & 0 & (40,2) & 3 & 4 \\
 \hline
 (1,14) & 0 & 0 & (1,12) & 1 & 0 & (11,2) & 1 & 0 & (64,2) & 3 & 0 \\
 \hline
 (2,2) & 0 & 0 & (1,16) & 1 & 0 & (14,2) & 1 & 0 & (38,2) & 4 & 3 \\
 \hline
 (2,4) & 0 & 0 & (1,18) & 1 & 0 & (17,2) & 1 & 0 & (44,2) & 4 & 0 \\
 \hline
 (2,6) & 0 & 0 & (1,20) & 1 & 0 & (19,2) & 1 & 0 & (56,2) & 5 & 0 \\
 \hline
 (4,2) & 0 & 0 & (1,22) & 1 & 0 & (20,2) & 1 & 0 & (67,2) & 5 & 1 \\
 \hline
 (4,4) & 0 & 0 & (1,26) & 1 & 0 & (32,2) & 1 & 0 & (80,2) & 7 & 0 \\
 \hline
 (5,2) & 0 & 0 & (2,8) & 1 & 0 & (49,2) & 1 & 0 & (140,2) & 19 & 0 \\
 \hline
 (7,2) & 0 & 0 & (2,10) & 1 & 0 & (2,12) & 2 & 63504 & (280,2) & 41 & 0 \\
 \hline
\end{array}
\end{equation}
\end{mytable}

Now, in the case of $m=2$, Clayton et al. \cite[Theorems 1.1 and 1.3]{clayton-et-al} already gave the complete list of $(N,k)$ for which $a_2(2,N,k)$ vanishes. But for completeness, we extend their result slightly and give the complete list of $(N,k)$ for which $a_2(2,N,k)$ is positive or zero. We forgo repeating all of the exact same details for the $m=2$ case, and just give the corresponding bounds that our method would yield.

\begin{proposition}[{c.f. \cite[Theorems 1.1 and 1.3]{clayton-et-al}}]
    Suppose that $\gcd(N,2)=1$ and that $k \geq 2$ is even.
    Then the second coefficient $a_2(2,N,k)$ is positive or zero only for the pairs $(N,k)$ given in Table \ref{table:T2-pairs-dim-a2}.
\end{proposition}

\begin{proof} 
    From Lemma \ref{lem:a2-coeff-formula}, we have
    \begin{equation} \label{eqn:a2-coeff-T2}
        a_2(2,N,k) = \frac{1}{2} \lrb{ (\Tr T_2)^2 - \Tr T_4 - 2^{k-1} \Tr T_1 }.
    \end{equation}
    Then by an identical method as in Lemmas \ref{lem:trace-T3-bound}, \ref{lem:trace-T9-bound}, and \ref{lem:trace-T1-bound1}, we obtain the bounds
    \begin{equation}
    \begin{alignedat}{6} 
        \frac{(\Tr T_2)^2}{\psi(N) 2^k } 
        &\leq  \frac{41 + 24\sqrt{2}}{4} \theta_2(N), &&&& \\
        \lrabs{\frac{\Tr T_4 - A_{1,4}}{\psi(N) 2^k}} 
        &\leq \frac{17}{2} \theta_3(N) + \frac{1}{4} \theta_1(N), && \qquad \text{where} \qquad && A_{1,4} = \frac{k-1}{48} \psi(N) 2^k , \\
        \lrabs{\frac{\Tr T_1 - A_{1,1}}{\psi(N)}} 
        &\leq \frac{5}{3} \theta_3(N) + \frac{1}{2} \theta_1(N), && \qquad \text{where} \qquad  && A_{1,1} = \frac{k-1}{12} \psi(N) .
    \end{alignedat}
    \end{equation}

    As done in Theorem \ref{thm:T3}, combining these three bounds gives    
    \begin{align}
         a_2(2,N,k) &= \frac{\psi(N) 2^k}{2} \lrb{- \frac{k-1}{16} + E(N,k) } ,
    \end{align}
    where 
    \begin{align}
        \lrabs{E(N,k)} &\leq  \frac{41 + 24\sqrt{2}}{4} \theta_2(N)  +  \frac{28}{3} \theta_3(N) + \frac{1}{2} \theta_1(N) . 
    \end{align}
    This yields $a_2(2,N,k) <0$ for $N \geq 2,\!700,\!000, k \geq 2$; $N \geq 150,\!000, k \geq 6$; $N \geq 8,\!800,\ k \geq 16$; $N \geq 571, k \geq 50$; $N \geq 43, k \geq 148$; and $N \geq 1, k \geq 562$. 

    We then check the finite number of cases left by computer, which yields the complete list given in Table \ref{table:T2-pairs-dim-a2}. See \cite{ross-code} for the code.
\end{proof}

\begin{mytable}
\label{table:T2-pairs-dim-a2}
\begin{equation}
\setlength{\arraycolsep}{2.5mm}
\begin{array}{|c|c|c||c|c|c||c|c|c||c|c|c|} 
 \hline
 \multicolumn{12}{|c|}{\makecell{
    \text{All pairs $(N,k)$ for which $a_2(2,N,k)$ is positive or zero,} \\  
    \text{ along with $\dim S_k(\Gamma_0(N))$ and the actual value of $a_2(2,N,k)$.}  
  }} \\
 \hline
 (N,k) & \dim & a_2 & (N,k) & \dim & a_2 & (N,k) & \dim & a_2 & (N,k) & \dim & a_2 \\
 \hline
 \hline
 (1,2) & 0 & 0 & (7,2) & 0 & 0 & (1,26) & 1 & 0 & (17,2) & 1 & 0 \\
 \hline
 (1,4) & 0 & 0 & (9,2) & 0 & 0 & (3,6) & 1 & 0 & (19,2) & 1 & 0 \\
 \hline
  (1,6) & 0 & 0 & (13,2) & 0 & 0 & (3,8) & 1 & 0 & (21,2) & 1 & 0 \\
 \hline
 (1,8) & 0 & 0 & (25,2) & 0 & 0 & (5,4) & 1 & 0 & (27,2) & 1 & 0 \\
 \hline
  (1,10) & 0 & 0 & (1,12) & 1 & 0 & (5,6) & 1 & 0 & (49,2) & 1 & 0 \\
 \hline
 (1,14) & 0 & 0 & (1,16) & 1 & 0 & (7,4) & 1 & 0 & (37,2) & 2 & 0 \\
 \hline
  (3,2) & 0 & 0 & (1,18) & 1 & 0 & (9,4) & 1 & 0 & (33,2) & 3 & 0 \\
 \hline
 (3,4) & 0 & 0 & (1,20) & 1 & 0 & (11,2) & 1 & 0 & (57,2) & 5 & 0 \\
 \hline
  (5,2) & 0 & 0 & (1,22) & 1 & 0 & (15,2) & 1 & 0 & \multicolumn{3}{c|}{~} \\
 \hline
\end{array}
\end{equation}
\end{mytable}

\section{Proof of Theorem \ref{thm:T4} } \label{sec:proofofT4}

In this section, we show Theorem \ref{thm:T4}: for $N$ coprime to $4$ and $k \geq 2$ even, $a_2(4,N,k)$ is negative or zero only for the pairs $(N,k)$ given in \cite[Table $m=4$]{ross-code}.

Observe that in the notation of Lemma \ref{lem:a2-coeff-formula}, for $\gcd(N,4)=1$ and $k \geq 2$ even, we have
\begin{equation} \label{eqn:a2-coeff-T4}
    a_2(4,N,k) = \frac{1}{2} \lrb{ (\Tr T_4)^2 - \Tr T_{16} - 2^{k-1} \Tr T_4 - 4^{k-1} \Tr T_1 }.
\end{equation}

We bound the four terms of \eqref{eqn:a2-coeff-T4} separately in Lemmas \ref{lem:trace-T4-bound1} - \ref{lem:trace-T1-bound2}.

\begin{lemma} \label{lem:trace-T4-bound1}
    Let $N \geq 1$ with $\gcd(N,4)=1$ and $k \geq 2$ be even such that $\frac{k-1}{48} \geq 7 \theta_3(N) + \frac14 \theta_1(N) $. 
    Then 
    \begin{align*}
        \frac{(\Tr T_4)^2}{(k-1) 4^k \psi(N)^2 } 
        &\geq \frac{k-1}{2304} - \frac{7}{24} \theta_3(N) - \frac{1}{96} \theta_1(N).
    \end{align*}
\end{lemma}
\begin{proof}
We examine each of the $A_{i,4}$ terms from \eqref{eqn:eichler-selberg-trace-formula} separately.

First, by \eqref{eqn:A1m-formula},
$$A_{1,4} = \chi_0(\sqrt{4}) \frac{k-1}{12} \psi (N) 4^{k/2-1} = \frac{k-1}{48} \psi (N) 2^k .$$

Second, summing over $t$ with $t^2<16$ and positive $n$ such that $n^2\mid (t^2-16)$ and $(t^2-16)/n^2\equiv 0,1\pmod{4}$, we get
\begin{align}
    A_{2,4} = &\frac{1}{2} \sum_{t^2 < 16} U_{k-1}(t,4) \sum_{n}  h_w \lrp{ \frac{t^2-16}{n^2} } \mu (t,n,4) \\
    = &\frac{1}{2} U_{k-1}(0,4) [ \mu(0,1,4) + \frac12 \mu(0,2,4) ] \\
    &+ U_{k-1}(1,4) \lrb{ 2 \mu (1,1,4) } \\
    &+ U_{k-1}(2,4) [ \mu(2,1,4) + \frac13 \mu(2,2,4) ] \\
    &+ U_{k-1}(3,4) \lrb{ \mu(3,1,4) }. 
\end{align}
So by Lemma \ref{lem:U-mu-bound},
\begin{align}
    \lrabs{A_{2,4}} &\leq  \, 2^{\omega(N)} 4^{(k-1)/2}\cdot 2 \lrb{\frac12 \lrp{1 + \frac{1}{2} \cdot 3} + (2) + \lrp{1+\frac{1}{3} \cdot 3} + \lrp{ 1 } } \\
    &= \frac{25}{4} \cdot 2^{k} \, 2^{\omega(N)} .
\end{align}

Third, by \eqref{eqn:A3m-formula} (recalling that the $d=d_0$ and $d=m/d_0$ terms coincide) and Lemma \ref{lem:sigma-N-m-d-bound}, 
\begin{align}
    A_{3,4} &= \frac{1}{2} \sum_{d|4} \min(d,4/d)^{k-1} \Sigma(N,4,d) \\
    &= \Sigma(N,4,1) + \frac12 \cdot 2^{k-1} \Sigma(N,4,2)  \\
    &\leq  |-3| \cdot 2^{\omega(N)} + \frac12 \cdot 2^{k-1} \sqrt{N} 2^{\omega(N)}\\
    &= 3\cdot 2^{\omega(N)} + \frac14\cdot  2^k \sqrt{N} 2^{\omega(N)}.
\end{align}

% Fourth, by \eqref{eqn:A4m-formula},
% \begin{align}
%    A_{4,4}  &\leq \sum_{c\mid 4} c \ =  1+2+4=7. 
% \end{align}

Putting this all together, and using the fact that $A_{4,16}\ge0$, we obtain 
\begin{align} 
    \Tr T_4 &= A_{1,4} - A_{2,4} - A_{3,4} + A_{4,4}  \\
    &\geq A_{1,4} - |A_{2,4}| - A_{3,4} \\
    &\geq \frac{k-1}{48} \psi(N) 2^k - \frac{25}{4}\cdot  2^k \,2^{\omega(N)} - 3 \cdot 2^{\omega(N)} - \frac14 \cdot 2^k \sqrt{N} 2^{\omega(N)} \\
    &\geq \frac{k-1}{48} \psi(N) 2^k - \frac{25}{4}\cdot  2^k \, 2^{\omega(N)} - \frac34\cdot  2^k \, 2^{\omega(N)} - \frac14 \cdot 2^k \sqrt{N} 2^{\omega(N)} \\
    &= \psi(N) 2^k \lrp{ \frac{k-1}{48} - 7 \theta_3(N) - \frac14 \theta_1(N) } .
\end{align}
Thus if $\frac{k-1}{48} \geq  7 \theta_3(N) + \frac14 \theta_1(N)$, then we have
\begin{align}
    \frac{(\Tr T_4)^2}{(k-1) 4^k \, \psi(N)^2} 
    &\geq   \frac{  \lrp{ \frac{k-1}{48} - 7 \theta_3(N) - \frac14 \theta_1(N) }^2  }{k-1} \\
    &\geq   \frac{  \lrp{ \frac{k-1}{48}}^2 - 2 \frac{k-1}{48} \lrp{ 7 \theta_3(N) + \frac14 \theta_1(N) }  }{k-1} \\
    &=   \frac{k-1}{2304} - \frac{7}{24} \theta_3(N) - \frac{1}{96} \theta_1(N),   \\
\end{align}
as desired.
\end{proof}

Next, we bound the $\Tr T_{16}$ term from \eqref{eqn:a2-coeff-T4}.

\begin{lemma} \label{lem:trace-T16-bound}
Let $N \geq 1$ with $\gcd(N,4)=1$ and $k \geq 2$ be even. Then
\begin{align}  
    \frac{\Tr T_{16}}{(k-1) 4^k \, \psi(N)^2}
    &\leq \frac{4309}{192N} .
\end{align}

\end{lemma}
\begin{proof}
We examine each of the $A_{i,16}$ terms from \eqref{eqn:eichler-selberg-trace-formula} separately.

First, by \eqref{eqn:A1m-formula},
\begin{align}
 A_{1,16} &= \chi_0(\sqrt{16}) \frac{k-1}{12} \psi (N) 16^{k/2-1} = \frac{k-1}{192} \psi (N) 4^k .
\end{align}

Second, summing over $t$ with $t^2<64$ and positive $n$ such that $n^2\mid (t^2-64)$ and $(t^2-64)/n^2\equiv 0,1\pmod{4}$, we get
\begin{align}
    A_{2,16} = &\frac{1}{2} \sum_{t^2 < 64} U_{k-1}(t,16) \sum_{n}  h_w \lrp{ \frac{t^2-64}{n^2} } \mu (t,n,16) \\
    % = &\frac{1}{2} U_{k-1}(0,16) \sum_{n}  h_w \lrp{ \frac{-64}{n^2} } \mu(0,n,16) \\
    %  &+U_{k-1}(1,16) \sum_{n}  h_w \lrp{ \frac{-63}{n^2} } \mu(1,n,16) \\
    %  &+U_{k-1}(2,16) \sum_{n}  h_w \lrp{ \frac{-60}{n^2} } \mu(2,n,16) \\
    %  &+U_{k-1}(3,16) \sum_{n}  h_w \lrp{ \frac{-55}{n^2} } \mu(3,n,16) \\
    %  &+U_{k-1}(4,16) \sum_{n}  h_w \lrp{ \frac{-48}{n^2} } \mu(4,n,16) \\
    %  &+U_{k-1}(5,16) \sum_{n}  h_w \lrp{ \frac{-39}{n^2} } \mu(5,n,16) \\
    %  &+U_{k-1}(6,16) \sum_{n}  h_w \lrp{ \frac{-28}{n^2} } \mu(6,n,16) \\
    %  &+U_{k-1}(7,16) \sum_{n}  h_w \lrp{ \frac{-15}{n^2} } \mu(7,n,16) \\
    %  = &\frac{1}{2} U_{k-1}(0,16) \lrb{ h_w(-64) \mu(0,1,16) + h_w(-16) \mu(0,2,16) + h_w(-4) \mu(0,4,16) }   \\
    %  &+U_{k-1}(1,16) \lrb{ h_w(-63) \mu(1,1,16) + h_w(-7) \mu(1,3,16) } \\
    %  &+U_{k-1}(2,16) \lrb{ h_w(-60) \mu(2,1,16) + h_w(-15) \mu(2,2,16) } \\
    %  &+U_{k-1}(3,16) \lrb{ h_w(-55) \mu(3,1,16) } \\
    %  &+U_{k-1}(4,16) \lrb{ h_w(-48) \mu(4,1,16) + h_w(-12) \mu(4,2,16) + h_w(-3) \mu(4,4,16) } \\
    %  &+U_{k-1}(5,16) \lrb{ h_w(-39) \mu(5,1,16)} \\
    %  &+U_{k-1}(6,16) \lrb{ h_w(-28) \mu(6,1,16) + h_w(-7) \mu(6,2,16) } \\
    %  &+U_{k-1}(7,16) \lrb{ h_w(-15) \mu(7,1,16) } \\
     = &\frac{1}{2} U_{k-1}(0,16) [ 2 \mu(0,1,16) + \mu(0,2,16) + \frac12 \mu(0,4,16) ]   \\
     &+U_{k-1}(1,16) \lrb{ 4 \mu(1,1,16) +  \mu(1,3,16) } \\
     &+U_{k-1}(2,16) \lrb{ 2 \mu(2,1,16) + 2 \mu(2,2,16) } \\
     &+U_{k-1}(3,16) \lrb{ 4 \mu(3,1,16) } \\
     &+U_{k-1}(4,16) [ 2 \mu(4,1,16) + \mu(4,2,16) + \frac13 \mu(4,4,16) ] \\
     &+U_{k-1}(5,16) \lrb{ 4 \mu(5,1,16)} \\
     &+U_{k-1}(6,16) \lrb{ \mu(6,1,16) + \mu(6,2,16) } \\
     &+U_{k-1}(7,16) \lrb{ 2 \mu(7,1,16) } .
\end{align}
Thus,  by Lemma \ref{lem:U-mu-bound},
\begin{align}
     |A_{2,16}| \leq & \, 2^{\omega(N)} 16^{(k-1)/2} \cdot 2 \Bigg[ \frac12 \lrp{2 + 3 + \frac12 \cdot 6} + (4 + 4) + (2 + 2\cdot 3) \\
     & \qquad\qquad\qquad\qquad + (4) + \lrp{2 + 3 + \frac13 \cdot 6} + (4) + \lrp{1+3} + (2) \Bigg] \\
     =&  \frac{41}{2}\cdot 2^{\omega(N)} 4^k .
\end{align}

Third, by \eqref{eqn:A4m-formula},
\begin{align}
   A_{4,16}  &\leq \sum_{c | 16} c \ =  1+2+4+8+16=31 .
\end{align}

Putting this all together, and using the fact that $A_{3,16}\ge0$, we obtain
\begin{align}
    \Tr T_{16}
    &= A_{1,16} - A_{2,16} - A_{3,16} + A_{4,16} \\
    &\leq  A_{1,16} + \lrabs{A_{2,16}} + A_{4,16} \\
    &\leq \frac{k-1}{192} \psi(N) 4^k + \frac{41}{2} \cdot2^{\omega(N)} 4^k + 31 \\
    &\leq \frac{k-1}{192} \psi(N) 4^k + \frac{41}{2} \psi(N) 4^k (k-1) + \frac{31}{16} \psi(N) 4^k (k-1) \\
    &=  \frac{4309}{192}\cdot \psi(N) 4^k (k-1).
\end{align}
This yields
\begin{align}
    \frac{\Tr T_{16}}{(k-1) 4^k \psi(N)^2} 
    \leq \frac{4309}{192 \psi(N)} 
    \leq \frac{4309}{192N},
\end{align}
as desired.
\end{proof}

\begin{lemma} \label{lem:trace-T4-bound2}
    Let $N \geq 1$ with $\gcd(N,4)=1$ and $k \geq 2$ be even. 
    Then 
    \begin{align*}
        \frac{\Tr T_4}{(k-1) 2^k \psi(N)^2 } 
        &\leq \frac{385}{48N}.
    \end{align*}
\end{lemma}
\begin{proof}
We examine each of the $A_{i,4}$ terms from \eqref{eqn:eichler-selberg-trace-formula} separately.

From the computations in the proof of Lemma \ref{lem:trace-T4-bound1}, we have
\begin{align} A_{1,4} = \frac{k-1}{48} \psi (N) 2^k \qquad
\text{and}\qquad 
\lrabs{A_{2,4}} &\leq  \frac{25}{4} 2^{k} \, 2^{\omega(N)} .
\end{align}

Also, by \eqref{eqn:A4m-formula},
\begin{align}
   A_{4,4}  &\leq \sum_{c\mid 4} c \ =  1+2+4=7. 
\end{align}

Putting all this together, and noticing that $A_{3,4}\ge0$, we obtain 
\begin{align} 
    \Tr T_4 &= A_{1,4} - A_{2,4} - A_{3,4} + A_{4,4}  \\
    &\leq A_{1,4} + |A_{2,4}| + A_{4,4} \\
    &\leq \frac{k-1}{48} \psi(N) 2^k + \frac{25}{4} \cdot 2^k \,2^{\omega(N)} + 7 \\
    &\leq \frac{k-1}{48} \psi(N) 2^k + \frac{25}{4}\cdot 2^k \psi(N)(k-1) + \frac{7}{4}\cdot 2^k \psi(N)(k-1) \\
    &=\frac{385}{48}\cdot  \psi(N) 2^k (k-1) .
\end{align}
This yields
\begin{align}
    \frac{\Tr T_4}{(k-1) 2^k \, \psi(N)^2} 
    \leq   \frac{  385  }{48 \psi(N)} 
    \leq   \frac{385}{48N} ,
\end{align}
as desired.
\end{proof}

\begin{lemma} \label{lem:trace-T1-bound2}
    Let $N \geq 1$ and $k \geq 2$ be even. Then
    $$\frac{\Tr T_1}{(k-1)\psi(N)^2} \leq \frac{9}{4N}. $$
\end{lemma}
\begin{proof}
    We examine each of the $A_{i,1}$ terms from \eqref{eqn:eichler-selberg-trace-formula} separately.
    
    By the calculations in the proof of Lemma \ref{lem:trace-T1-bound1}, we have  
    \begin{align}
        A_{1,1} = \frac{k-1}{12} \psi (N), \qquad
        \lrabs{A_{2,1}} \leq \frac{7}{6} 2^{\omega(N)} ,\qquad  
    \text{and} \qquad
    A_{4,1} \leq   1 . \end{align}

    Putting all this together, and noticing that $A_{3,1}\ge0$, we obtain
    \begin{align}
        \Tr T_1 &= A_{1,1} - A_{2,1} - A_{3,1} + A_{4,1} \\
        &\leq  A_{1,1}  + \lrabs{A_{2,1}} + A_{4,1}  \\
        &\leq  \frac{k-1}{12} \psi(N) + \frac{7}{6} 2^{\omega(N)} + 1 \\
        &\leq  \frac{k-1}{12} \psi(N) + \frac{7}{6} \psi(N)(k-1) + \psi(N)(k-1) \\
        &=  \frac94 \cdot \psi(N) (k-1) ,
    \end{align}
    which yields
    \begin{align}
        \frac{\Tr T_1}{(k-1)\psi(N)^2}
        \leq \frac{9}{4\psi(N)} 
        \leq \frac{9}{4N},
    \end{align}
    completing the proof.   
\end{proof}

We are now ready to prove Theorem \ref{thm:T4}.
{
\renewcommand{\thetheorem}{\ref{thm:T4}}
\begin{theorem}
    Suppose that $\gcd(N,4)=1$ and that $k \geq 2$ is even.
    Then  $a_2(4,N,k)$ is negative or zero only for the pairs $(N,k)$ given in \cite[Table $m=4$]{ross-code}.
\end{theorem}
\addtocounter{theorem}{-1}
}
\begin{proof} 
    Assume that the condition $\frac{k-1}{48} \geq 7 \theta_3(N) + \frac14 \theta_1(N)$ from Lemma \ref{lem:trace-T4-bound1} is satisfied. Then by \eqref{eqn:a2-coeff-T4} and Lemma \ref{lem:trace-T4-bound1},
    \begin{align}
        &\ \ \ \  a_2(4,N,k)\\
        &= \frac{1}{2} \lrb{ (\Tr T_4)^2 - \Tr T_{16} - 2^{k-1} \Tr T_4 - 4^{k-1} \Tr T_1 } \\ 
        &= \frac{(k-1)4^k \psi(N)^2}{2} 
        \lrb{ 
            \frac{(\Tr T_4)^2}{ (k-1)4^k \psi(N)^2 } 
            - 
            \frac{\Tr T_{16}}{ (k-1)4^k \psi(N)^2 } - 
            \frac12 \frac{\Tr T_4}{(k-1)2^k \psi(N)^2} 
            - 
            \frac14 \frac{\Tr T_1}{ (k-1) \psi(N)^2}
        } \\ 
        &\geq \frac{(k-1)4^k \psi(N)^2}{2} \Bigg[
            \frac{k-1}{2304} - \frac{7}{24} \theta_3(N) - \frac{1}{96} \theta_1(N)
            - 
            \frac{\Tr T_{16}}{ (k-1)4^k \psi(N)^2 } \\
            & \qquad\qquad\qquad\qquad\qquad\qquad
            - 
            \frac12 \frac{\Tr T_4}{(k-1)2^k \psi(N)^2} 
            - 
            \frac14 \frac{\Tr T_1}{ (k-1) \psi(N)^2} \Bigg] \\
        &= \frac{(k-1)4^k \psi(N)^2}{2} \lrb{ \frac{k-1}{2304} - E(N,k) }, 
        \label{eqn:a2-4-final-eqn}
    \end{align}
    where $E(N,k)$ denotes the five error terms above. In particular, from Lemmas \ref{lem:trace-T16-bound}, \ref{lem:trace-T4-bound2}, and \ref{lem:trace-T1-bound2},
    \begin{align}
        E(N,k) &= 
        \frac{7}{24} \theta_3(N) 
        + \frac{1}{96} \theta_1(N) 
        + \frac{\Tr T_{16}}{ (k-1)4^k \psi(N)^2 }
        + \frac12 \frac{\Tr T_4}{(k-1)2^k \psi(N)^2} 
        + \frac14 \frac{\Tr T_1}{ (k-1) \psi(N)^2} \\
        &\leq
        \frac{7}{24} \theta_3(N) 
        + \frac{1}{96} \theta_1(N)
        + \frac{4309}{192N} 
        + \frac12 \cdot \frac{385}{48N}
        + \frac14 \cdot \frac{9}{4N} \\
        &= \frac{7}{24} \theta_3(N) 
        + \frac{1}{96} \theta_1(N)
        + \frac{1729}{64N}.
    \end{align}

    For $N \geq 2,\!700,\!000$, we have the bounds  $\theta_1(N) \leq 0.0265$ and $\theta_3(N) \leq 0.000015$ given in Lemma \ref{lem:omega-over-psi-bounds}. Note for such $N$, the condition $\frac{k-1}{48} \geq 7 \theta_3(N) + \frac14 \theta_1(N)$ from Lemma \ref{lem:trace-T4-bound1} is satisfied for all $k \geq 2$. 
    Additionally, 
    \begin{align}
        E(N,k) &\leq \frac{7}{24} \cdot 0.000015 
        + \frac{1}{96} \cdot 0.0265
        + \frac{1729}{64 \cdot 2,\!700,\!000}  \\
        &\leq 0.000291,
    \end{align}
    which is $< \frac{k-1}{2304}$ for all $k \geq 2$. By \eqref{eqn:a2-4-final-eqn}, this shows that $a_2(4,N,k) > 0$ for $N \geq 2,\!700,\!000$ and $k \geq 2$. 

    Utilizing the table in Lemma \ref{lem:omega-over-psi-bounds}, an identical argument using $150,\!000$, $8,\!800$, $571$, $43$, and $1$ as the bounds for $N$, shows that $a_2(4,N,k) > 0$ for $N \geq 150,\!000, k \geq 4$; $N \geq 8,\!800,\ k \geq 14$; $N \geq 571, k \geq 124$; $N \geq 43, k \geq 1498$; and $N \geq 1, k \geq 62942$. 

    We then check the finite number of cases left by computer, which yields the complete list given in \cite[Table $m=4$]{ross-code}. (The table contains $164$ pairs, and was too large to include here.)  See \cite{ross-code} for the code. 
\end{proof}

Now, the same methods here would also work for the perfect square case of $m=1$. But this would just be proving a result that is already known. $T_1(N,k)$ is just the identity operator, so $a_2(1,N,k)$ will be positive precisely for the pairs $(N,k)$ for which $\dim S_k(\Gamma_0(N)) \geq 2$. And the complete list of pairs $(N,k)$ where $\dim S_k(\Gamma_0(N)) = 0 \text{ or }1$, was already computed in \cite[Tables 2.6, 2.7]{ross}

% \begin{mytable}
% \label{table:T3-a2-coeff-vanishing}  
% \begin{equation}
% \setlength{\arraycolsep}{4mm}
% \begin{array}{|c|c|c|c|c|c|c|c|} 
%  \hline
%  \multicolumn{8}{|c|}{\text{The pairs $(N,k)$ for which $a_2(3,N,k)$ vanishes}} \\
%  \hline
%  (1,2) & (1,4) & (1,6) & (1,8) & (1,10) & (1,12) & (1,14) & (1,16) \\ 
%  \hline
%  (1,18) & (1,20) & (1,22) & (1,26) & (2,2) & (2,4) & (2,6) & (2,8) \\
%  \hline
%  (2,10) & (4,2) & (4,4) & (4,6) & (5,2) & (5,4) & (5,6) & (7,2) \\
%  \hline
%  (7,4) & (8,2) & (8,4) & (10,2) & (11,2) & (13,2) & (14,2) & (16,2) \\
%  \hline
%  (17,2) & (19,2) & (20,2) & (25,2) & (32,2) & (34,2) & (44,2) & (49,2) \\
%  \hline
%  (56,2) & (64,2) & (80,2) & (140,2) & (280,2) & \multicolumn{3}{c|}{~} \\
%  \hline
% \end{array}
% \end{equation}

% \end{mytable}

% \vtop{\hbox{\strut top line}\hbox{\strut botline}}

% \begin{mytable}
% \label{table:T2-a2-coeff-vanishing}
% \begin{equation}
% \setlength{\arraycolsep}{4mm}
% \begin{array}{|c|c|c|c|c|c|c|c|} 
%  \hline
%  \multicolumn{8}{|c|}{\text{The pairs $(N,k)$ for which $a_2(2,N,k)$ vanishes}} \\
%  \hline
%  (1,2) & (1,4) & (1,6) & (1,8) & (1,10) & (1,12) & (1,14) & (1,16) \\ 
%  \hline
%  (1,18) & (1,20) & (1,22) & (1,26) & (3,2) & (3,4) & (3,6) & (3,8) \\
%  \hline
%  (5,2) & (5,4) & (5,6) & (7,2) & (7,4) & (9,2) & (9,4) & (11,2) \\
%  \hline
%  (13,2) & (15,2) & (17,2) & (19,2) & (21,2) & (25,2) & (27,2) & (33,2) \\
%  \hline
%  (37,2) & (49,2) & (57,2) & \multicolumn{5}{c|}{~} \\
%  \hline
% \end{array}
% \end{equation}
% 
% \end{mytable}

\section{Discussion} \label{sec:discussion}
We first make some observations about the proofs of Propositions \ref{prop:main-theorem-not-perfect-square} and \ref{prop:main-theorem-perfect-square}. One of the last steps in these two proofs was to show that for any given $(m,N,\chi)$, the $a_2$ coefficient is nonvanishing for sufficiently large $k$. It is worth noting that is also what one would expect from the well-known Skolem-Mahler-Lech theorem \cite[Section 2.1]{Everest2003}. For any fixed $(m,N,\chi)$, $\Tr T_m(N,k,\chi)$ satisfies a linear recurrence in $k$; this fact follows from the Eichler-Selberg trace formula \eqref{eqn:eichler-selberg-trace-formula}. So by Lemma \ref{lem:a2-coeff-formula}, $a_2(m,N,k,\chi)$ also satisfies a linear recurrence in $k$. Thus by Skolem-Mahler-Lech, one would expect that for any given $(m,N,\chi)$, $a_2(m,N,k,\chi)$ only vanishes for finitely many $k$. 

Now, one might ask if the same result holds including $m$ as well, i.e. if $a_2(m,N,k,\chi)$ only vanishes for finitely many $(m,N,k,\chi)$. It turns out this is not true for a trivial reason. When $\dim S_k(\Gamma_0(N),\chi) < 2$, the characteristic polynomial for $T_m(N,k,\chi)$ will have degree less than $2$. Thus the second coefficient will neccessarily be $0$ for all $m$ in this case.

So one might instead ask if $a_2(m,N,k,\chi)$ vanishes only for finitely many $(m,N,k,\chi)$ when we restrict to the case of $\dim S_k(\Gamma_0(N),\chi) \geq 2$. It turns out this is not true either. For example, $S_2(\Gamma_0(23))$ has dimension $2$. And the two Hecke eigenforms for $S_2(\Gamma_0(23))$ both happen to have $43$rd Fourier coefficient vanish \cite{lmfdb-23-2-a-a}. Thus for any $m$ with exactly one factor of $43$, these two eigenforms will both have $m$-th Fourier coefficient vanish. So for such $m$, $T_{m}(23,2)$ will have two eigenvalues of $0$, and hence will have characteristic polynomial $x^2$. This gives an infinite family of vanishing $a_2$ coefficients.

One could also ask if the result still holds if we allow $N$ not coprime to $m$. In fact, this appears to not be true either. For example, it appears that for $j \geq 2$, all the eigenvalues of $T_3(3^j,4)$ vanish. This would give an infinite family of $N$ for which $a_2(3,N,4)$ vanishes. One can find many other similar examples. 

Next, we make some remarks about Table \ref{table:T3-pairs-dim-a2}, Table \ref{table:T2-pairs-dim-a2}, and \cite[Table $m=4$]{ross-code}.
In principle, one could run the same computations for general character $\chi$ and calculate the finitely many triples $(N,k,\chi)$ for which $a_2(m,N,k,\chi)$ vanishes.
However, the number of $(N,k,\chi)$ to check makes this computationally infeasible.
Even with only considering trivial character, we used the Clemson University PALMETTO cluster, and it took several hours to check all the possible pairs $(N,k)$.

Additionally, observe that many of the entries in these three tables come from $(N,k)$ for which $\dim S_k(\Gamma_0(N))$ is $0$ or $1$ (so the $a_2$ coefficient trivially vanishes).
For $m=2$, $32$ out of $35$ entries in Table \ref{table:T2-pairs-dim-a2} come from $(N,k)$ for which $\dim S_k(\Gamma_0(N)) = 0 \text{ or }1$. All of the remaining $3$ nontrivial entries have $a_2$ coefficient vanish.
For $m=3$, $38$ out of $52$ entries in Table \ref{table:T3-pairs-dim-a2} come from $(N,k)$ for which $\dim S_k(\Gamma_0(N)) = 0 \text{ or }1$. Of the remaining $14$ nontrivial entries, $7$ have positive $a_2$ coefficient and $7$ have vanishing $a_2$ coefficient.
For $m=4$, $32$ out of $164$ entries in \cite[Table $m=4$]{ross-code} come from $(N,k)$ for which $\dim S_k(\Gamma_0(N)) = 0 \text{ or }1$. Of the remaining $132$ nontrivial entries, $129$ have negative $a_2$ coefficient and $3$ have vanishing $a_2$ coefficient.

Looking at the $(N,k)$ given in these three tables, one might ask if $a_2(m,N,k)$ is non-trivially vanishing (i.e. is vanishing when $\dim S_k(\Gamma_0(N)) \geq 2$) only for $k = 2$. In fact, this is not true: for $m=7, N=12, k=4$, $a_2(7,12,4) = 0$. However this is the only counter-example that we could find, and we speculate that it is the only such counter-example.

One might also ask if the $a_2$ coefficient never non-trivially vanishes in the level one case. In fact, a more general result was conjectured in \cite{clayton-et-al}.
\begin{conjecture}[{\cite[Conjecture~5.1]{clayton-et-al}}]
    Let $m \geq 1$, $k \geq 2$ be even, and $n = \dim S_k(\Gamma_0(1))$.
    Write the characteristic polynomial for $T_m(1,k)$ as 
    $$T_m(1,k)(x) = x^n - a_1(m,1,k) x^{n-1} + a_2(m,1,k) x^{n-2} - \ldots + (-1)^n a_n(m,1,k).$$
    Then $a_i(m,1,k) \neq 0$ for $1 \leq i \leq n$.
\end{conjecture}

\bibliographystyle{plain}
\bibliography{bibliography.bib}

\end{document}